\documentclass[11pt, reqno]{amsart}
\usepackage{amssymb}
\DeclareMathAlphabet{\mathsl}{OT1}{cmr}{m}{sl}

\numberwithin{equation}{section}

\newcommand{\grD}{\mathsf{D}}

\newcommand{\grI}{\mathsf{I}}

\newcommand{\grV}{\mathsf{V}}
\newcommand{\grW}{\mathsf{W}}
\newcommand{\grY}{\mathsf{Y}}
\newcommand{\grA}{\mathsf{A}}
\newcommand{\grK}{\mathsf{K}}
\newcommand{\id}{\operatorname{\mathsl{id}}}
\newcommand{\an}{\text{\textsl{an}}}
\renewcommand{\min}{\operatornamewithlimits{\mathsl{min}}}

\renewcommand{\deg}{\operatorname{\mathsl{deg}}}
\renewcommand{\dim}{\operatorname{\mathsl{dim}}}
\renewcommand{\ker}{\operatorname{\mathsl{ker}}}

\newcommand{\al}{\alpha}
\newcommand{\rest}[1]{\rvert_{#1}}
\newcommand{\grof}[2]{\gr_{#1}({#2})}
\newcommand{\wh}{\widehat}
\newcommand{\ov}{\overline}
\newcommand{\Vlam}{V_\lambda}
\newcommand{\grC}{\mathsf{C}}
\newcommand{\grT}{\mathsf{T}}

\renewcommand{\int}{\operatorname{\mathsl{int}}}
\renewcommand{\tilde}{\widetilde}

\DeclareMathOperator{\End}{\operatorname{\mathsl{End}}}
\DeclareMathOperator{\grEnd}{\operatorname{\mathsf{End}}}

\DeclareMathOperator{\grHom}{\operatorname{\mathsf{Hom}}}
\DeclareMathOperator{\gr}{\operatorname{\mathsf{gr}}}
\DeclareMathOperator{\Nrd}{\operatorname{\mathsl{Nrd}}}

\DeclareMathOperator{\im}{\operatorname{\mathsl{im}}}
\DeclareMathOperator{\Gal}{\operatorname{\mathcal{G}}}

\DeclareMathOperator{\charac}{\operatorname{\mathsl{char}}}
\DeclareMathOperator{\ad}{\operatorname{\mathsl{ad}}}

\DeclareMathOperator{\ms}{\operatorname{\mathsl{ms}}}
\DeclareMathOperator{\Sym}{\operatorname{\mathsl{Sym}}}
\DeclareMathOperator{\Symd}{\operatorname{\mathsl{Symd}}}
\DeclareMathOperator{\rk}{\operatorname{\mathsl{rk}}}
\DeclareMathOperator{\trdeg}{\operatorname{\mathsl{trdeg}}}
\DeclareMathOperator{\Tr}{\operatorname{\mathsl{Tr}}}

\def\hsp{\mspace{1mu}}
\newcommand{\DIM}[2]{[#1{\hsp:\hsp}#2]}
\newcommand{\IND}[2]{\lvert#1{\hsp:\hsp}#2\rvert}

\def\tsum{\textstyle\sum\limits}

\def\tbigoplus{\textstyle\bigoplus\limits}

\newcommand{\jdotfont}{}
\font\jdotfont lcircle10  scaled 913 
\newcommand{\osbullet}{\jdotfont\char113} 
\newlength{\sbwd} 
\newcommand{\csbullet}{\kern.5\sbwd\osbullet\kern-.5\sbwd}

\newcommand{\joinrelshort}{\mathrel{\mkern-9mu}}
\newcommand{\shortlongrightarrow}{\relbar\joinrelshort\rightarrow}
\newcommand{\iso}{\mathrel{\mathop{\setbox0\hbox{$\mathsurround0pt
           \shortlongrightarrow$}\ht0=0.7\ht0\box0}\limits
           ^{\sim\mkern2mu}}}

\newtheorem{proposition}{Proposition}[section]
\newtheorem{theorem}[proposition]{Theorem}
\newtheorem{corollary}[proposition]{Corollary}
\newtheorem{lemma}[proposition]{Lemma}

\theoremstyle{definition}
\newtheorem{definition}[proposition]{Definition}

\theoremstyle{remark}
\newtheorem*{remark}{Remark}
\newtheorem{remarks}[proposition]{Remarks}
\newtheorem{example}[proposition]{Example}

\title{Valuations on algebras with involution}
\author{J.-P. Tignol \and A.R. Wadsworth}
\address{Institut de Math\'ematique Pure et Appliqu\'ee\\
Universit\'e catholique de Louvain\\
B-1348 Louvain-la-Neuve\\
Belgium}
\email{jean-pierre.tignol@uclouvain.be}
\address{Department of Mathematics\\
University of California, San Diego\\
La Jolla, CA  92093-0112\\
USA}
\email{arwadsworth@ucsd.edu}
\thanks{The first author is partially supported by the F.R.S.--FNRS 
(Belgium).  The second author would like to thank the first author 
and UCL for their hospitality during several visits while this paper 
was developing}

\begin{document}
\maketitle

Valuations are a major tool for the study of the structure of division
algebras. The purpose of this work is to introduce a notion that plays
a similar role for central simple algebras with involution, and
to prove analogues for this notion to fundamental results on
valuations on division algebras.

Since the definition of a (Schilling) valuation implies an
absence of zero divisors, the only central simple algebras that can 
have valuations are division algebras.
Given a division algebra $D$ finite-dimensional over its center $F$, it
is natural to view valuations on a  $D$ as extensions of valuations on  $F$,
since valuations on fields are abundant and their theory is well-developed.
But not every valuation~$v$ on~$F$ extends to $D$.
In the extension question,
Henselian valuations play a special role.  Schilling proved in 
\cite[pp.~ 53--54]{Sch} that if $v$ on $F$ is Henselian, then $v$ has 
an extension to a valuation on $D$,
and this extension is unique.  Much later it was proved by
Ershov \cite{Er0} and Wadsworth \cite{W} that for any
valuation $v$ on $F$, $v$ extends to $D$ if and only if it satisfies a
Henselian-like condition with respect to the field extensions
of $F$ within $D$; they also proved that when $v$ extends to
$D$ the extension is unique.  Another fundamental criterion was
proved by Morandi \cite{M}:  $v$~on $F$ extends to a valuation on $D$ 
if and only if $D$~remains a division algebra after scalar extension 
to the Henselization $F_h$ of $F$ for $v$.
We will prove analogues for central simple algebras with involutions 
to these theorems of Schilling, Ershov-Wadsworth, and Morandi.

An involution on a central simple algebra $A$ is a 
ring-anti-automorphism $\sigma$ such that $\sigma^2=\id_A$. As Weil 
suggested in~\cite{We}, the theory of central simple algebras with
involution is a natural sibling to the theory of central
simple algebras, since the associated automorphism groups are the
basic types of classical groups.  In each setting there is
a notion of anisotropic object, corresponding to when the
associated automorphism group is anisotropic as an algebraic
group.  The anisotropic central simple algebras are the division 
algebras.  An involution $\sigma$ on a central simple algebra $A$ is 
anisotropic just when the equation
$\sigma(x) x = 0$ holds only for $x=0$.  In earlier work
\cite{TWgr} we have developed the theory of gauges, which are
a kind of value functions for central simple algebras. (The
definition  of a gauge is recalled at the end of this introduction.) 
For a central simple algebra
$A$ with involution $\sigma$, we define a
{\it $\sigma$-special gauge}
to be a gauge $\varphi$  on~$A$
satisfying the
condition\footnote{Notice the similarity with the
    definition of $C^*$-algebras, cf.~\cite[D\'ef.~1.3.1]{Dix}.}
that $\varphi(\sigma(x)x) = 2\varphi(x)$
for all $x\in A$.  A $\sigma$-special gauge for an algebra with involution
is our analogue to a valuation on a division algebra.  If $A$
has a $\sigma$-special gauge, then $\sigma$ is easily seen to be anisotropic.
If $v$~is a Henselian valuation on the $\sigma$-invariant part
$F$ of $Z(A)$ and $\sigma$
is anisotropic, we show in Th.~\ref{mainHensel.thm} that there is a 
unique $\sigma$-invariant gauge $\varphi$ on $A$ extending
$v$, and $\varphi$ is a $\sigma$-special gauge.  When $v$ on $F$
is not Henselian, we show in Th.~\ref{mainnonHensel.thm} that
there is a $\sigma$-special gauge $\varphi$ on $A$ extending
$v$ if and only if the anisotropic
involution $\sigma$ remains anisotropic after scalar extension to the 
Henselization of $F$ with respect to $v$; furthermore,
there is  only one such $\varphi$.   Our results require
tame ramification and exclude orthogonal involutions if the residue 
characteristic is $2$; see the statements of
Theorems~\ref{mainHensel.thm} and \ref{mainnonHensel.thm}
for the precise conditions required.

A gauge $\varphi$ on a central simple algebra $A$ induces a 
filtration on $A$ which yields an associated graded ring
$\gr(A)$, analogous to what one has with a valuation on a field or a 
division ring.  The graded structure is intrinsic to the definition 
of a gauge, and is used heavily throughout this paper.  The 
degree~$0$ part of $\gr(A)$, denoted $A_0$, is the
residue ring of the \lq\lq valuation ring" of $A$ determined by
the gauge $\varphi$;  $A_0$~is~always a semisimple
$Z(A)_0$-algebra,
but not simple  in general.  If $\sigma$ is an involution on $A$ and 
$\varphi$~is invariant under $\sigma$, then $\sigma$ induces 
involutions $\tilde \sigma$ on $\gr(A)$
and $\sigma_0$ on $A_0$.  We show in Prop.~\ref{eqcond.prop} and
Remark~\ref{anisot.rem}(1) that a $\sigma$-invariant
gauge $\varphi$ is  $\sigma$-special  if and only if
$\tilde \sigma$ is anisotropic, if and only if $\sigma_0$ is anisotropic.
We also prove an analogue of a theorem
of Springer: when the base field is Henselian, an involution
$\sigma$ is
isotropic if and only if its residue involution $\sigma_0$
is isotropic
(Cor.~\ref{anisot.cor}). This criterion is applied to show that under
specified valuation-theoretic conditions, an anisotropic involution
remains anisotropic after certain scalar extensions (Cor.~\ref{cor:isotcrit}).

An outline of this paper is as follows: In Section~\ref{sec:compa},
we discuss in general terms the compatibility of a value function with
an involution, relating that notion to a compatibility condition
between norms and hermitian forms defined in \cite{RTW}. In
Section~\ref{sec:Hensel}, we restrict to the case of Henselian
valuations and give the proofs of Th.~\ref{mainHensel.thm} and
Cor.~\ref{anisot.cor}. Some applications to scalar extensions (in
particular Cor.~\ref{cor:isotcrit}) are given in
Section~\ref{tens.sec}. Sections~\ref{comp.sec} and \ref{descent.sec}
prepare the ground for the extension of our results to the
non-Henselian case in Section~\ref{nonHensel.sec}. The main problem is
to analyze how the condition for the existence of a splitting base of
a value function (which is a critical part of the definition of a
gauge) behaves under restriction of scalars; this is done in
Section~\ref{descent.sec}. In Section~\ref{comp.sec}, we investigate
this condition for the composition of value functions. This is used in
Section~\ref{nonHensel.sec} in the proof of
Th.~\ref{mainnonHensel.thm} by induction on the rank of valuations.
\medbreak
\par
For the convenience of the reader, we now review the basic notions of 
value functions, norms, and gauges
introduced in \cite{RTW} and \cite{TWgr}. Throughout the paper, we fix
a divisible totally ordered  abelian group $\Gamma$, which will contain
the values of all the valuations and the degrees of all the gradings
we consider. Thus, a valued field $(F,v)$ is a pair consisting of a
field $F$ and a valuation $v\colon F\to\Gamma\cup\{\infty\}$. The
group $v(F^\times)$ of values of $F$ is denoted by~$\Gamma_F$, and the
residue field by $\overline{F}$. We use  analogous notation for
valuations on division rings.

Let $(F,v)$ be a valued field. A \emph{$v$-value function} on an
$F$-vector space $V$ is a map $\alpha\colon V\to\Gamma\cup\{\infty\}$
such that
\begin{enumerate}
\item[(i)]
$\alpha(x)=\infty$ if and only if $x=0$;
\item[(ii)]
$\alpha(x+y)\geq\min\bigl(\alpha(x),\alpha(y)\bigr)$ for $x$, $y\in V$;
\item[(iii)]
$\alpha(xc)=\alpha(x)+v(c)$ for  all $x\in V$ and $c\in
F$.
\end{enumerate}
The $v$-value function $\alpha$ is called a \emph{norm} if $V$
is finite-dimensional and
contains a base $(e_i)_{i=1}^n$ such that
\[
\alpha\bigl(\tsum_{i=1}^ne_ic_i\bigr) \ = \
\min\limits_{1\leq i\leq n}\bigl(\alpha(e_ic_i)\bigr)
\qquad\text{for $c_1$, \ldots, $c_n\in F$}.
\]
Such a base is called a \emph{splitting base} of $V$ for $\alpha$. A
$v$-value function $\varphi$ on an $F$-algebra $A$ is
\emph{surmultiplicative} if $\varphi(1)=0$ and
$\varphi(xy)\geq\varphi(x)+\varphi(y)$ for $x$, $y\in A$.

The valuation $v$ defines a filtration on $F$: for $\gamma\in\Gamma$
we set
\[
F^{\geq\gamma}\, = \,\{x\in F\mid v(x)\geq\gamma\},\quad
F^{>\gamma} \, = \, \{x\in F\mid v(x)>\gamma\},\quad
\text{and }
F_\gamma \, = \, F^{\geq\gamma}/F^{>\gamma}.
\]
The associated graded ring is
\[
\gr(F) \ = \ \tbigoplus_{\gamma\in\Gamma}F_\gamma.
\]
It is called a \emph{graded field} because every nonzero homogeneous
element in $\gr(F)$ is invertible. Likewise, every $v$-value function
$\alpha$ on an $F$-vector space $V$ defines a filtration, and the
associated graded structure $\gr_\alpha(V)$ is a graded module over
$\gr(F)$, which we call a \emph{graded vector space}. It is a free
module, whose rank is called its dimension. The value function is a norm
if and only if $\dim_{\gr(F)}(\gr_\alpha(V))=\dim_F(V)<\infty$, see
\cite[Cor.~2.3]{RTW}. Every nonzero element $x\in V$ has an image
$\tilde x$ in $\gr_\alpha(V)$ defined by
\[
\tilde x \, = \,  x+V^{>\alpha(x)} \, \in V_{\alpha(x)}.
\]
We also set $\tilde 0=0\in\gr_\alpha(V)$. If $\varphi$ is a
surmultiplicative $v$-value function on an $F$-algebra $A$, then
$\gr_\varphi(A)$ is an algebra over $\gr(F)$, in which multiplication
is defined by
\[
\tilde a\tilde b \, \,= \,\, ab+V^{>\varphi(a)+\varphi(b)} \ = \
\begin{cases}
\tilde{ab}&\text{if $\varphi(ab)=\varphi(a)+\varphi(b)$},\\
0&\text{if $\varphi(ab)>\varphi(a)+\varphi(b)$},
\end{cases}
\quad\text{for $a$, $b\in A$}.
\]

Now, suppose $A$ is a finite-dimensional simple $F$-algebra. We denote
by $\DIM{A}{F}$ its dimension and by $Z(A)$ its center. A
surmultiplicative $v$-value function $\varphi$ on $A$ is called a
\emph{$v$-gauge} if it satisfies the following conditions:
\begin{enumerate}
\item[(i)]
$\varphi$ is a $v$-norm, i.e.,
$\DIM{A}{F}=\DIM{\gr_\varphi(A)}{\gr(F)}$;
\item[(ii)]
$\gr_\varphi(A)$ is a graded semisimple $\gr(F)$-algebra, i.e., it does
not contain any nonzero nilpotent homogeneous two-sided ideal.
\end{enumerate}
The $v$-gauge $\varphi$ is said to be  \emph{tame} if
$Z\bigl(\gr_\varphi(A)\bigr)=\gr_\varphi\bigl(Z(A)\bigr)$ and
$Z\bigl(\gr_\varphi(A)\bigr)$ is separable over $\gr(F)$.
  If the residue characteristic is $0$,
then every $v$-gauge is tame, see~\cite[Cor.~3.6]{TWgr}.

\section{Special gauges}
\label{sec:compa}

Let $(F,v)$ be a valued field and let $A$ be an
$F$-algebra. An $F$-linear involution on $A$ is an $F$-linear map
$\sigma\colon A\to A$ such that
\begin{enumerate}
\item[(i)]
$\sigma(x+y)=\sigma(x)+\sigma(y)$ for $x$, $y\in A$;
\item[(ii)]
$\sigma(xy)=\sigma(y)\sigma(x)$ for $x$, $y\in A$;
\item[(iii)]
$\sigma^2(x)=x$ for $x\in A$.
\end{enumerate}
(The $F$-linearity implies that $\sigma\rvert_F=\id_F$.)
A surmultiplicative $v$-value function $\varphi\colon
A\to\Gamma\cup\{\infty\}$ is
said to be  \emph{invariant under $\sigma$} if
\begin{equation}\label{compdef}
\varphi\bigl(\sigma(x)\bigr)\, =\, \varphi(x)\qquad\text{for all $x\in A$}.
\end{equation}
The involution then preserves the filtration on $A$ defined by
$\varphi$. Therefore, it induces an involution $\tilde\sigma$ on
$\gr_\varphi(A)$ such that
\[
\tilde\sigma(\tilde x) \, = \, \tilde{\sigma(x)}\qquad\text{for all $x\in A$}.
\]

As in \cite[\S6.A]{BoI}, we say that the involution $\sigma$ is
\emph{anisotropic} if there is no nonzero element $x\in A$ such that
$\sigma(x)x=0$. Likewise, $\tilde\sigma$ is said to be anisotropic if
there is no nonzero homogeneous element $\xi\in\gr_\varphi(A)$ such
that $\tilde\sigma(\xi)\xi=0$.  Clearly, if $\tilde\sigma$ is
anisotropic, then $\sigma $ is anisotropic.

\begin{proposition}
       \label{eqcond.prop}
       Let $\varphi$ be a surmultiplicative $v$-value function and
       $\sigma$ an $F$-linear involution on $A$.
        The following conditions are
       equivalent:
       \begin{enumerate}
	\item[(a)]
	$\varphi(\sigma(x)x)=2\varphi(x)$ for all $x\in A$;
	\item[(b)]
	$\varphi$ is invariant under $\sigma$, and $\tilde\sigma$ is
           anisotropic.
        \end{enumerate}
        They imply that if $x$, $y\in A$ satisfy $\sigma(x)y=0$ or
        $x\sigma(y)=0$, then
        \begin{equation}
	 \label{sumorth.eq}
	 \varphi(x+y) \, = \, \min\bigl(\varphi(x),\varphi(y)\bigr).
        \end{equation}
        Moreover, when these equivalent conditions hold,
$\sigma$ is anisotropic and the
        $\gr(F)$-algebra $\gr_\varphi(A)$  contains no nonzero
        homogeneous nil left or right ideal.
\end{proposition}

\begin{proof}
       (a)~$\Rightarrow$~(b):
If $\sigma(x)x= 0$, then condition (a) implies that $\varphi(x) = 
\infty$, so $x=0$.  Thus, $\sigma$ is anisotropic.
       By surmultiplicativity, we have
       \[
       \varphi(\sigma(x)x) \, \geq  \,
       \varphi(\sigma(x))+\varphi(x)\qquad\text{for all $x\in A$}.
       \]
       Therefore, (a) implies $\varphi(x)\geq \varphi(\sigma(x))$ for all
       $x\in A$.
       Substituting $\sigma(x)$ for $x$ in this inequality, we obtain
       $\varphi(\sigma(x))\geq \varphi(x)$ for all $x\in A$. Therefore,
       $\varphi$ is invariant under $\sigma$, and condition (a) can be
       reformulated as
       $\varphi(\sigma(x)x)=\varphi\bigl(\sigma(x)\bigr)+\varphi(x)$ for
       all $x\in A$. Thus, it implies
       \[
       \tilde{\sigma(x)}\tilde x \, = \,
(\sigma(x)x)^\sim\qquad\text{for all $x\in A$},
       \]
       whence $\tilde\sigma$ is anisotropic, as $\sigma$ is
anisotropic.

       (b)~$\Rightarrow$~(a):
For all $x\in A$
       we have
       \[
       \tilde\sigma(\tilde x)\tilde x \ = \
       \begin{cases}
	(\sigma(x)x)^\sim&\text{if
             $\varphi(\sigma(x)x)=\varphi\bigl(\sigma(x)\bigr)+\varphi(x)$,}\\
	0&\text{if
             $\varphi(\sigma(x)x)>\varphi\bigl(\sigma(x)\bigr)+\varphi(x)$}.
       \end{cases}
       \]
Condition (b) implies that the
first case always occurs.  Hence, for all $x$, $\ {\varphi(\sigma(x)x) =
\varphi(\sigma(x)) + \varphi(x) = 2\varphi(x)}$.

       For the rest of the proof, assume (a) and (b) hold.
Then clearly $\sigma$ is anisotropic. Also,
for $x$,
       $y\in A$ we have by surmultiplicativity
       \begin{equation}
	\label{jsig.eq}
	\varphi\bigl(\sigma(x)\cdot(x+y)\bigr) \, \geq \,
	\varphi\bigl(\sigma(x)\bigr)+\varphi(x+y) \, = \,
           \varphi(x)+\varphi(x+y).
       \end{equation}
       If $\sigma(x)y=0$, then
       \begin{equation}
	\label{jsig2.eq}
	\varphi\bigl(\sigma(x)\cdot(x+y)\bigr) \, = \, \varphi(\sigma(x)x)
    \, = \, 2\varphi(x).
       \end{equation}
       By combining \eqref{jsig.eq} and \eqref{jsig2.eq}, we obtain
       $\varphi(x)\geq \varphi(x+y)$. Similarly, interchanging $x$ and 
$y$ we get
       $\varphi(y)\geq \varphi(x+y)$, hence
       \[
       \min\bigl(\varphi(x),\varphi(y)\bigr) \, \geq \,  \varphi(x+y).
       \]
       The reverse inequality holds by definition of a value function,
       hence \eqref{sumorth.eq} is proved when $\sigma(x)y=0$. If
       $x\sigma(y)=0$, we substitute $\sigma(x)$ for $x$ and $\sigma(y)$
       for $y$ in the arguments above, obtaining
       \[
       \varphi\bigl(\sigma(x)+\sigma(y)\bigr) \, = \,
       \min\bigl(\varphi(\sigma(x)),
       \varphi(\sigma(y))\bigr).
       \]
       Equation~\eqref{sumorth.eq} follows since $\varphi\circ\sigma=\varphi$.

       To complete the proof, suppose $\grI\subset\gr_\varphi(A)$ is a
       homogeneous nil left (resp.\ right) ideal and $\xi\in\grI$ is a
       nonzero homogeneous element. Let $\eta=\tilde\sigma(\xi)\xi$
       (resp.\ $\eta=\xi\tilde\sigma(\xi)$). Then $\eta\in\grI$ is
       $\tilde\sigma$-symmetric, homogeneous, and nonzero since
       $\tilde\sigma$ is anisotropic. Since $\grI$ is nil, we may find
       $k\geq1$ such that $\eta^k\neq0$ and $\eta^{k+1}=0$. For
       $\zeta=\eta^k$ we have
       \[
       \tilde\sigma(\zeta)\zeta \, = \, \zeta^2 \, = \, \eta^{2k} \, = \, 0,
       \]
       so $\zeta=0$, a contradiction.
\end{proof}

\begin{definition}\label{specialdef}
A surmultiplicative $v$-value function $\varphi$ on a central simple 
algebra $A$ with involution~$\sigma$ is called {\it
$\sigma$-special} if it satisfies the
conditions (a) and (b) of Prop.~\ref{eqcond.prop}.
\end{definition}

For use in \S\S\ref{tens.sec} and \ref{nonHensel.sec}, we record how
involution invariance of value functions
 behaves with respect to tensor products.
Recall from \cite[Prop.~1.23, (1.25)]{TWgr} that if $V$ is a finite-dimensional
$F$-vector space with a $v$-norm $\alpha$ and $W$ is an $F$-vector space with
$v$-value function $\beta$, then there is a $v$-value function
$\alpha \otimes \beta$ on $V\otimes _F W$ uniquely determined by the
condition that the map $(x\otimes y)^\sim\mapsto\widetilde{x}\otimes
\widetilde{y}$ (for $x\in V$ and $y\in W$) defines an isomorphism of
graded vector spaces
\begin{equation}
\label{eq:griso}
\Omega\colon \gr_{\alpha\otimes\beta}(V\otimes_F W)\iso
\gr_\alpha(V)\otimes_{\gr(F)}\gr_\beta(W).
\end{equation}
In particular,
\begin{equation}\label{alphatensorbeta}
(\alpha\otimes\beta)(x\otimes y)=\alpha(x)+\beta(y)
\quad \text{for all $x\in V$ and $y\in W$}.
\end{equation}
The value function $\alpha\otimes\beta$ can be defined as follows:
take any splitting base $(e_i)_{i=1}^n$ for $\alpha$ on $V$; then,
\[
(\alpha \otimes \beta)\bigl( \, \tsum_{i=1}^n e_i\otimes y_i \, \bigr)
    \ = \ \min\limits_{1\le i\le n}\bigl(\alpha(e_i) + \beta(y_i)\bigr)
\quad \text{for
any $y_1,\ldots,y_n \in W$}.
\]
Furthermore, analogous to \cite[Cor.~1.26]{TWgr},  if $(W, \beta)$ is
a valued field extending $(F,v)$, then
$\alpha\otimes \beta$ is a $\beta$-norm on $V\otimes_F W$ and
\eqref{eq:griso} is a $\gr_\beta(W)$-vector space isomorphism.

\begin{proposition}
     \label{prop:tensprod}
     Let $\sigma$ and $\tau$ be $F$-linear involutions on $F$-algebras
     $A$ and $B$ respectively, and let $\varphi$ $($resp.~{$\psi$}$)$ be a
     surmultiplicative $v$-value function on $A$
$($resp.~{$B$}$)$ invariant
     under $\sigma$ $($resp.~{$\tau$}$)$. Suppose $A$~is finite-dimensional and
     $\varphi$ is a $v$-norm. Then, $\varphi\otimes\psi$ is a
     surmultiplicative $v$-value function on $A\otimes_FB$ invariant
     under the involution $\sigma\otimes\tau$, and the canonical
     isomorphism $\Omega$ of \eqref{eq:griso} is an isomorphism of graded
     $\gr(F)$-algebras with involution,
     \[
     \bigl(\gr_{\varphi\otimes\psi}(A\otimes_FB),
     \widetilde{\sigma\otimes\tau}\bigr) \iso
     \bigl(\gr_\varphi(A)\otimes_{\gr(F)}\gr_\psi(B),
     \widetilde{\sigma}\otimes\widetilde{\tau}\bigr).
     \]
\end{proposition}

\begin{proof}
     Let $(e_i)_{i=1}^n$ be a splitting base of $A$ for $\varphi$. For
     $x$, $y\in A\otimes_FB$ we may write
     \[
     x \, = \, \tsum_{i=1}^n e_i\otimes x_i\qquad\text{and}\qquad
     y \, = \, \tsum_{j=1}^n e_j\otimes y_j\qquad\text{for some $x_1$, \ldots,
       $y_n\in B$}.
     \]
     Then,
     \[
       (\varphi\otimes\psi)(xy)  \ = \
       (\varphi\otimes\psi)\bigl(\tsum_{i,j}e_ie_j\otimes x_ix_j\bigr)
        \ \geq \
       \min\limits_{1\leq i,j\leq n}\bigl(
       (\varphi\otimes\psi)(e_ie_j\otimes x_iy_j)\bigr)
        \ = \  \min\limits_{i,j}\bigl(\varphi(e_ie_j) +\psi(x_iy_j)\bigr).
     \]
     Since $\varphi$ and $\psi$ are surmultiplicative, we have
     $\varphi(e_ie_j)\geq\varphi(e_i)+\varphi(e_j)$ and
     $\psi(x_iy_j)\geq\psi(x_i)+\psi(y_j)$, hence
     \[
     (\varphi\otimes\psi)(xy) \ \geq \ \min\limits_{i,j}\bigl(
     \varphi(e_i)+\varphi(e_j) +\psi(x_i)+\psi(y_j)\bigr)  \ = \
     \min\limits_i\bigl(\varphi(e_i)+\psi(x_i)\bigr) +
     \min\limits_j\bigl(\varphi(e_j)+\psi(y_j)\bigr).
     \]
     The right side is $(\varphi\otimes\psi)(x)+(\varphi\otimes\psi)(y)$,
     so
     \[
     (\varphi\otimes\psi)(xy) \ \geq \ (\varphi\otimes\psi)(x)+
     (\varphi\otimes\psi)(y).
     \]
     Since moreover
     $(\varphi\otimes\psi)(1\otimes1)=\varphi(1)+\psi(1)=0$,
     surmultiplicativity of $\varphi\otimes\psi$ is proved.

To show that the $gr_F$-vector space isomorphism $\Omega$
is a ring isomorphism, we check this for $\Omega^{-1}$.
The canonical $F$-algebra homomorphisms $\iota_A\colon
A \to A\otimes_FB$, $a \mapsto a\otimes 1$ and
$\iota_B\colon B \to A\otimes_F B$, $b \mapsto 1\otimes b$,
are value-preserving. Hence, they induce $\gr(F)$-algebra
homomorphisms $\tilde\iota_A\colon
\gr_\al(A) \to \gr_{\al\otimes \beta}(A\otimes_FB)$,
$\tilde a \mapsto \tilde{a\otimes 1}$, and
$\tilde\iota_B\colon \gr_\beta(B) \to
\gr_{\al \otimes \beta}(A\otimes_F B)$,
$\tilde b \mapsto \tilde{1\otimes b}$.  For any
$a\in A$ and $b\in B$, we have from \eqref{alphatensorbeta},
$$
\tilde{a\otimes 1}\cdot \tilde{1\otimes b} \ = \
[(a\otimes 1)\cdot(1\otimes b)]^\sim
\ = \
\tilde{a\otimes b} \ = \
[(1\otimes b) \cdot (a\otimes 1)]^\sim
   \ = \ \tilde{1\otimes b} \cdot \tilde{a\otimes 1}.
$$
Thus, $\im(\tilde\iota_B)$ centralizes $\im(\tilde\iota_A)$
in $\gr_{\alpha\otimes\beta}(A\otimes_F B)$. So, there
is an induced $\gr(F)$-algebra homomorphism
$\gr_\alpha(A) \otimes _{\gr(F)} \gr_\beta(B) \to
\gr_{\al \otimes \beta}(A\otimes_F B)$ given by
$\tilde a \otimes \tilde b \, \mapsto \,  \tilde{a\otimes 1}
\cdot \tilde{1\otimes b}  \, = \,  \tilde{a\otimes b}$.
The description of $\Omega$ preceding \eqref{eq:griso}
shows that this algebra homomorphism is $\Omega^{-1}$.

     To prove $\varphi\otimes\psi$ is invariant under
     $\sigma\otimes\tau$, we first show that $(\sigma(e_i))_{i=1}^n$ also
     is a splitting base of $A$ for $\varphi$.
     Take any $c_1$, \ldots, $c_n\in F$. Then, as $\varphi$ is invariant
     under $\sigma$ and the $c_i$ are central in $A$ and fixed under $\sigma$,
     \[
     \varphi\big(\tsum_{i=1}^n \sigma(e_i)c_i\big) \ = \
     \varphi\big(\sigma\big(\tsum_{i=1}^n \sigma(e_i)c_i\big) \big)\ = \
     \varphi\big(\tsum_{i=1}^n e_ic_i\big) \\
     = \
     \min\limits_{1\le i\le n}\big(\varphi(e_i) +v(c_i) \big)
     \ =\ \min\limits_{1\le i\le n}\big(\varphi(\sigma(e_i)) +v(c_i) \big).
     \]
     Thus, $(\sigma(e_i))_{i=1}^n$ is a splitting base for
     $\varphi$. With the notation above, we then have
     \[
     (\varphi\otimes\psi)\bigl((\sigma\otimes\tau)(x)\bigr) \ = \
     (\varphi\otimes\psi)\bigl(\tsum_{i=1}^n\sigma(e_i)\otimes\tau(x_i)\bigr)
      \ = \  \min\limits_{1\leq i\leq
       n}\bigl(\varphi(\sigma(e_i))+\psi(\tau(x_i))\bigr).
     \]
     Since $\varphi$ is invariant under $\sigma$ and $\psi$ under $\tau$,
     we have
     \[
     \min\limits_{1\leq i\leq
       n}\bigl(\varphi(\sigma(e_i))+\psi(\tau(x_i))\bigr) \ = \
     \min\limits_{1\leq i\leq n}\bigl(\varphi(e_i)+\psi(x_i)\bigr)  \ = \
     (\varphi\otimes\psi)(x).
     \]
     Therefore, $\varphi\otimes\psi$ is invariant under
     $\sigma\otimes\tau$. To complete the proof, observe that for $a\in
     A$ and $b\in B$ we have
     \[
     \widetilde{\sigma\otimes\tau}(\widetilde{a\otimes b})  \, =\,
     (\sigma(a)\otimes\tau(b))^\sim \, = \, (\widetilde{\sigma}\otimes
     \widetilde{\tau})(\widetilde{a}\otimes\widetilde{b}),
     \]
     hence the involution $\widetilde{\sigma}\otimes\widetilde{\tau}$
     corresponds to $\widetilde{\sigma\otimes\tau}$ under the canonical
     isomorphism~\eqref{eq:griso}.
\end{proof}

The following special case will be particularly useful:

\begin{corollary}
     \label{cor:compatext}
     Let $A$ be a finite-dimensional $F$-algebra with an $F$-linear
     involution $\sigma$ and let $(K,v_K)$ be any valued field extension
     of $(F,v)$. If $\varphi$ is a surmultiplicative $v$-norm on $A$
     which is invariant under~$\sigma$, then $\varphi\otimes v_K$ is a
     surmultiplicative $v_K$-norm on $A\otimes_FK$ invariant under the
     involution $\sigma\otimes\id_K$, and
     $(\sigma\otimes\id_K)^\sim\cong
     \widetilde{\sigma}\otimes\id_{\gr(K)}$ under the canonical
     isomorphism~\eqref{eq:griso}.
\end{corollary}

\begin{proof}
     It suffices to note that $\widetilde{\id_K}=\id_{\gr(K)}$ and that
     the canonical isomorphism is an isomorphism of $\gr(K)$-algebras,
     see \cite[Cor.~1.26]{TWgr}.
\end{proof}

Now, assume $A$ is simple and finite-dimensional, and let $n=\deg A$, so
$\DIM{A}{Z(A)}=n^2$. Recall from \cite{BoI} that involutions on $A$
are classified into two kinds and three types: an involution $\sigma$ is
of the \emph{first kind} if $\sigma\rvert_{Z(A)}=\id_{Z(A)}$; otherwise it is
of the \emph{second kind}. Involutions of the second kind are also said to be
\emph{of unitary type} (or simply \emph{unitary}). To define the type
of an involution $\sigma$ of the first kind we consider the subspaces
of symmetric and of symmetrized elements in $A$, defined by
\[
\Sym(A,\sigma) \, = \, \{x\in A\mid\sigma(x)=x\}
\qquad\text{and}\qquad
\Symd(A,\sigma) \, = \, \{x+\sigma(x)\mid x\in A\}.
\]
The involution $\sigma$ is \emph{of symplectic type} (or simply
\emph{symplectic}) if either $\charac (F)\neq2$ and
$\dim_{Z(A)}\Sym(A,\sigma)=\frac12n(n-1)$ or $\charac (F)=2$ and
$1\in\Symd(A,\sigma)$. Involutions of the first kind that are not
symplectic are said to be \emph{of orthogonal type} (or simply
\emph{orthogonal}). If $\sigma$ is orthogonal, then
$\dim_{Z(A)}\Sym(A,\sigma)=\frac12n(n+1)$. The same terminology is
used for involutions on graded simple algebras.

\begin{proposition}
\label{type.prop}
Let $\sigma$ be an $F$-linear involution on a finite-dimensional
simple $F$-algebra $A$
and let $g$ be a tame $v$-gauge on $A$ that is invariant under
$\sigma$. Suppose $F$ is the subfield of $Z(A)$ fixed under $\sigma$.

If $\sigma$ is unitary, two cases may arise:
\begin{itemize}
\item[--]
if the valuation $v$ extends uniquely from $F$ to $Z(A)$, then
$\gr_g(A)$ is a graded simple $\gr(F)$-algebra and $\tilde\sigma$ is a
unitary involution;
\item[--]
if the valuation $v$ has two different extensions to $Z(A)$, then
$\gr_g(A)$ is a direct product of two graded central simple
$\gr(F)$-algebras, which are exchanged under $\tilde\sigma$.
\end{itemize}

If $\sigma$ is symplectic, then $\tilde\sigma$ is a symplectic
involution on the graded central simple $\gr(F)$-algebra $\gr_g(A)$.

If $\sigma$ is orthogonal and $\charac(\overline{F})\neq2$, then
$\tilde\sigma$ is an orthogonal involution on the graded central
simple $\gr(F)$-algebra $\gr_g(A)$.
\end{proposition}

\begin{proof}
Suppose first that $\sigma$ is unitary, so $Z(A)/F$ is a quadratic
extension. By \cite[Cor.~2.5]{TWgr}, the number of simple components
of $\gr(A)$ equals the number of extensions of $v$ to $Z(A)$. Therefore,
to complete the description of $\tilde\sigma$ it suffices to show that
$\tilde\sigma$ does not identically fix
$Z\bigl(\gr(A)\bigr)=\gr\bigl(Z(A)\bigr)$.
Since  the Galois group $\Gal(Z(A)/F)$ acts transitively
on the set of extensions of $v$ to $Z(A)$, see \cite[Th.~3.2.15,
p.~64]{EP}, if there are two such extensions, then $\sigma\rest {Z(A)}$
must permute
them; then $\tilde \sigma$ permutes  the corresponding components of
$\gr(Z(A))$. So, we may assume that $v$ has a unique extension to
$Z(A)$.
Then, $\gr\big(Z(A)\big)$ is a graded field separable over
$\gr(F)$, and $\DIM{\gr\big(Z(A)\big)}{\gr(F)}=
\DIM{Z(A)}F = 2$ since $g\rest {Z(A)}$ is a norm,
by \cite[Prop.~2.5]{RTW}.
If $\charac
(\overline{F})\neq2$ we can find $z\in Z(A)$ nonzero such that
$\sigma(z)=-z$, hence $\tilde\sigma(\tilde z)=-\tilde z\neq\tilde
z$. If $\charac (\overline{F})=2$ the separability of 
$\gr\bigl(Z(A)\bigr)$ over
$\gr(F)$ implies by \cite[Th.~3.11, Def.~3.4]{HW1} that
$\Gamma_{Z(A)} = \Gamma_F$ and $Z(A)_0$ is separable over $F_0$;
furthermore, $\DIM{Z(A)_0}{F_0}= \DIM{Z(A)}F = 2$ since
$g\rest {Z(A)}$ is a norm, by \cite[Prop.~2.5]{RTW}.
So $Z(A)$ is unramified Galois over $F$, hence
the non-trivial automorphism
$\sigma\rvert_{Z(A)}$ induces a non-trivial automorphism of the
residue algebra $Z(A)_0$, by \cite[Th.~19.6,
p.~124]{E}, showing that
$\tilde\sigma\rest{\gr(Z(A))}$ is nontrivial.

Suppose next that $\sigma$ is of the first kind, so $Z(A)=F$. For $x\in A$
we have $\tilde x+\tilde\sigma(\tilde x)=\bigl(x+\sigma(x)\bigr)^\sim$
~or~$0$. On the other hand, $\sigma(x)=x$ implies $\tilde\sigma(\tilde
x)=\tilde x$. Therefore, the following inclusions are clear:
\begin{equation}\label{grsubsets}
\gr\bigl(\Sym(A,\sigma)\bigr) \, \subseteq  \, \Sym(\gr(A),\tilde\sigma),
\qquad
\Symd(\gr(A),\tilde\sigma) \, \subseteq \, \gr\bigl(\Symd(A,\sigma)\bigr).
\end{equation}
If $\charac(\overline{F})\neq2$ (hence $\charac (F)\neq2$) we have
\[
\Sym(A,\sigma) \, = \, \Symd(A,\sigma)\qquad\text{and}\qquad
\Sym(\gr(A),\tilde\sigma) \, = \, \Symd(\gr(A),\tilde\sigma),
\]
so the inclusions in \eqref{grsubsets} above yield
$\gr\bigl(\Sym(A,\sigma)\bigr)=\Sym(\gr(A),\tilde\sigma)$. Since the
type of an involution can be determined from the dimension of the
space of symmetric elements, it follows that $\tilde\sigma$ has the
same type as $\sigma$.

To complete the proof, suppose $\charac(\overline{F})=2$ and $\sigma$ is
symplectic, and let $n=\deg A$. Since $\tilde\sigma$ is of the first
kind we have
\[
\dim_{\gr(F)}\Symd(\gr(A),\tilde\sigma) \ = \ {\textstyle\frac12}n(n-1).
\]
On the other hand, since $\sigma$ is symplectic we have
\[
\dim_F\Symd(A,\sigma) \, = \, {\textstyle\frac12}n(n-1)
\]
(independently of whether $\charac (F)=2$). Since $g$ is a norm we
have
\[
\dim_F\Symd(A,\sigma) \ = \ \dim_{\gr(F)}\gr\bigl(\Symd(A,\sigma)\bigr),
\]
hence
$\Symd(\gr(A),\tilde\sigma)=\gr\bigl(\Symd(A,\sigma)\bigr)$. Since
$1\in\Symd(A,\sigma)$, it follows that
$\tilde1\in\Symd(\gr(A),\tilde\sigma)$, hence $\tilde\sigma$ is
symplectic.
\end{proof}

\begin{remark}
If $\sigma$ is orthogonal and $\charac(\overline{F})=2$, the involution
$\tilde\sigma$ may be symplectic, as the following example shows: let
$(F,v)$ be a valued field with $\charac (F)=0$ and
$\charac(\overline{F})=2$, and let $A=M_2(F)$. Define an orthogonal
involution $\sigma$ on $A$ by
\[
\textstyle
\sigma
\left(\begin{smallmatrix}
a&b\\
c&d
\end{smallmatrix}
\right)
    \, = \,
\left(
\begin{smallmatrix}
d&b\\
c&a
\end{smallmatrix}
\right)
\]
and a $v$-gauge $g$ by
\[
g\left(
\begin{smallmatrix}
a&b\\
c&d
\end{smallmatrix}
\right)
    \ = \
\min\bigl(v(a),v(b),v(c),v(d)\bigr).
\]
This gauge is clearly invariant under $\sigma$. We have
$\gr_g(A)=M_2\bigl(\gr(F)\bigr)$ with the entrywise grading, and
\[
\left(
\begin{smallmatrix}
1&0\\
0&1
\end{smallmatrix}
\right) \ = \
\left(\begin{smallmatrix}
1&0\\
0&0
\end{smallmatrix}
\right)
+ \,
\tilde\sigma\left(
\begin{smallmatrix}
1&0\\
0&0
\end{smallmatrix}
\right)\, \in \, \Symd(\gr_g(A),\tilde\sigma).
\]
Therefore, $\tilde\sigma$ is symplectic.

Henceforth, we systematically avoid orthogonal involutions in
characteristic~$2$.
\end{remark}

In \cite[\S3]{RTW}, a notion of compatibility is defined between norms and
hermitian forms. In the rest of this section, we relate that notion
    of compatibility with the invariance of value functions under involutions.

Let $D$ be a finite-dimensional division $F$-algebra with an $F$-linear
involution $\tau$. Suppose $v$ extends to a valuation
$w$ on $D$ invariant under $\tau$ and let $V$ be a finite-dimensional
right $D$-vector space. Consider a nondegenerate hermitian form
$h\colon V\times V\to D$ with respect to $\tau$, and a $w$-norm
$\alpha$ on $V$. The dual norm $\alpha^\sharp$ is defined by
\begin{equation}\label{dualnormdef}
\alpha^\sharp(x) \ = \ \min\{\,
w\bigl(h(x,y)\bigr)-\alpha(y)\mid y\in V,\;y\neq0\}
\qquad\text{for $x\in V$},
\end{equation}
see \cite[\S3]{RTW}. The norm $\alpha$ is said to be
\emph{compatible with} $h$ if and
only if $\alpha^\sharp=\alpha$ (see \cite[Prop.~3.5]{RTW}).
This is the condition needed in order for $h$ to induce a
nondegenerate graded hermitian form on $\gr_\alpha(V)$.
On the simple algebra $\End_D(V)$
there is   the involution $\ad_h$
adjoint to $h$, defined by
\[
h(\ad_h(f)(x),y) \ =  \ h(x, f(y))  \quad\quad \text{for all } x,y\in V.
\]
    There is also  the well-defined surmultiplicative $v$-value function
$\End(\alpha)$  on $\End_D(V)$ defined by
\[
\End(\alpha)(f) \ = \ \min\{\alpha\bigl(f(x)\bigr)-\alpha(x)\mid x\in
V,\;x\neq0\}.
\]
Recall that $\End(\alpha)$ is a $v$-gauge if and only if $w$ on $D$ is
defectless over $v$,
see \cite[Prop.~1.19]{TWgr}.

\enlargethispage{\baselineskip}
\enlargethispage{\baselineskip}

\begin{proposition}
\label{compat.prop}
The value functions $\End(\alpha)$ and $\End(\alpha^\sharp)$ are
related by
\begin{equation}
\label{compat.eq}
\End(\alpha)\circ\ad_h \ = \ \End(\alpha^\sharp).
\end{equation}
Moreover, the following conditions are equivalent:
\begin{enumerate}
\item[(a)]
$\End(\alpha)$ is invariant under $\ad_h$;
\item[(b)]
$\End(\alpha^\sharp) \, = \, \End(\alpha)$;
\item[(c)]
$\alpha-\alpha^\sharp$ is constant on $V$;
\item[(d)]
there is a constant $\gamma\in\Gamma$ such that $\alpha-\gamma$ is
compatible with $h$.
\end{enumerate}
\end{proposition}

\begin{proof}
Let $(e_i)_{i=1}^n$ be a splitting base of $V$ for $\alpha$. The
$h$-dual base $(e_i^\sharp)_{i=1}^n$ for $V$ is a splitting base  for~
$\alpha^\sharp$, by \cite[Lemma~3.4]{RTW}. Fix some $f\in\End_DV$, and
let
\[
f(e_j^\sharp) \, = \,  \tsum_{i=1}^ne_i^\sharp d_{ij}\qquad\text{for some
     $d_{ij}\in D$}.
\]
Then, computation yields
\[
\ad_h(f)(e_j) \ = \ \tsum_{i=1}^ne_i\tau(d_{ji}).
\]
We may compute $\End(\alpha^\sharp)(f)$ using the splitting base
$(e_i^\sharp)_{i=1}^n$, and $\End(\alpha)\bigl(\ad_h(f)\bigr)$ using
the splitting base $(e_i)_{i=1}^n$, obtaining
\begin{align*}
\End(\alpha^\sharp)(f)&  \ = \
    \min_{1\leq i,j\leq n}\bigl(\alpha^\sharp(e_i^\sharp)+
w(d_{ij})-\alpha^\sharp(e_j^\sharp)\bigr),\\
\End(\alpha)\bigl(\ad_h(f)\bigr)&  \ = \
\min_{1\leq i,j\leq n}\bigl(\alpha(e_i)+w\bigl(\tau(d_{ji})\bigr)
-\alpha(e_j)\bigr).
\end{align*}
Equation~\eqref{compat.eq} follows since
$\alpha^\sharp(e_i^\sharp)=-\alpha(e_i)$, see \cite[Lemma~3.4]{RTW}.

The equivalence of (a) and (b) readily follows from \eqref{compat.eq},
and the equivalence of (b) and (c) from \cite[~Prop.~1.22]{TWgr}.
\par\noindent
(c)~$\Leftrightarrow$~(d):
By the definition of the dual norm in \eqref{dualnormdef}, for any constant
$\gamma$ in the divisible group  $\Gamma$, $(\alpha - \gamma)^\sharp
= \alpha^\sharp + \gamma$. Therefore, $\alpha - \gamma$ is compatible
with $h$ if and only if $(\alpha -\gamma)^\sharp = \alpha -\gamma$,
which holds if and only if $\alpha - \alpha^\sharp = 2\gamma$.
\end{proof}

Suppose the equivalent conditions of Prop.~\ref{compat.prop}
hold, and write simply $g_\alpha$ for $\End(\alpha)$. Recall from
\cite[Prop.~1.19]{TWgr} that the graded algebra $\gr_{g_\alpha}(\End_DV)$
may be identified with $\End_{\gr(D)}\bigl(\gr_\alpha(V)\bigr)$ so
that for $f\in \End_DV$ the element $\tilde
f\in\gr_{g_\alpha}(\End_DV)$ is viewed as the map $\tilde
f\colon\gr_\alpha(V)\to\gr_\alpha(V)$ defined by
\[
\tilde f(\tilde x) \ = \
\begin{cases}
\tilde{f(x)}&\text{if
     $\alpha\bigl(f(x)\bigr)=\alpha(x)+g_\alpha(f)$},\\
0&\text{if $\alpha\bigl(f(x)\bigr)>\alpha(x)+g_\alpha(f)$}.
\end{cases}
\]
On the other hand, after adding a constant if necessary,
we may assume $\alpha$ is compatible with
$h$; hence we may define a graded hermitian form
\[
\tilde {h}\colon\gr_\alpha(V)\times\gr_\alpha(V) \, \to \, \gr_w(D)
\]
(with respect to the involution $\tilde\tau$) as follows: for $x$,
$y\in V$,
\[
\tilde {h}(\tilde x,\tilde y) \, = \, h(x,y)+D^{>\alpha(x)+\alpha(y)}
\ =\
\begin{cases}
\tilde{h(x,y)}&\text{if $w\bigl(h(x,y)\bigr)=\alpha(x)+\alpha(y)$},\\
0&\text{if $w\bigl(h(x,y)\bigr)>\alpha(x)+\alpha(y)$}.
\end{cases}
\]
This hermitian form is well-defined and
nondegenerate (see \cite[Remark~3.2]{RTW}), so we
may consider the adjoint involution $\ad_{\tilde {h}}$  on
$\End_{\gr_w(D)}\bigl(\gr_\alpha(V)\bigr)=\gr_{g_\alpha}(\End_DV)$.

\begin{proposition}
\label{adjoint.prop}
Assuming $\alpha$ is compatible with $h$, the involution $\tilde\ad_h$
on $\gr_{g_\alpha}(\End_DV)$ is the adjoint involution of $\tilde
{h}$ under the identification above; i.e.,
\[
\tilde\ad_h \, = \, \ad_{\tilde {h}}.
\]
In particular, the value function $g_\alpha=\End(\alpha)$
is $\sigma$-special $($see Def.~\ref{specialdef}$)$
  if and only if
$\tilde {h}$ is anisotropic.
\end{proposition}

\begin{proof}
To verify the equality of graded involutions, it suffices to show, for all
$x$, $y\in V$ and $f\in\End_DV$,
\[
\tilde {h}\bigl(\ad_{\tilde {h}}(\tilde f)(
\tilde x),\tilde y\bigr) \ = \
\tilde {h}\bigl(\tilde\ad_h(\tilde f)(\tilde x),\tilde y\bigr).
\]
   From the definition of $\ad_{\tilde {h}}$, it is equivalent to prove
\begin{equation}\label{ad.eq}
\tilde {h}\bigl(\tilde x,\tilde f(\tilde y)\bigr) \ = \
\tilde {h}\bigl(\tilde\ad_h(\tilde f)(\tilde x),\tilde y\bigr).
\end{equation}
Since $\alpha$ is compatible with $h$,  Prop.~\ref{compat.prop}
shows
$g_\alpha$ is invariant under $\ad_h$; hence, ${g_\alpha(\ad_h(f))
= g_\alpha(f)}$. Therefore, each side of \eqref{ad.eq} lies in $D_\epsilon$\,,
where $\epsilon = \alpha(x) + \alpha (y) + g_\alpha(f)$.  Suppose
${w\bigl(h(x,f(y))\bigr) = \epsilon}$.  Then, necessarily $\alpha(f(y)) =
g_\alpha(f) + \alpha(y)$, and the left side of \eqref{ad.eq} equals
$h(x, f(y))^\sim$.  But since ${h(x, f(y)) = h\bigl(\ad_h(f)(x),y\bigr)}$,
we then
also have $g_\alpha\big(\ad_h(f)(x)\big) =g_\alpha\big(\ad_h(f)\big)+
\alpha (x)$,
and the right side of \eqref{ad.eq} becomes
$h\bigl(\ad_h(f)(x),y\bigr)^\sim\,$.
So, \eqref{ad.eq} then holds.   But, if $w\big(h(x,f(y))\big)>\epsilon$, then
each side of~\eqref{ad.eq} is $0$.  Thus, the equality \eqref{ad.eq} holds
in all cases, so that $\tilde\ad_h \, = \, \ad_{\tilde {h}}$.

Since $g_\alpha$ is invariant under $\ad_h$,
Prop.~\ref{eqcond.prop}(b) holds if and only if
$\tilde{\ad}_h$  is anisotropic.  But, the involution ${\tilde{\ad}_h =
\ad_{\tilde {h}}}$ is anisotropic if and only if its associated graded
hermitian form $\tilde h$ is anisotropic.  This is proved analogously to the
ungraded case \cite[\S6.A]{BoI},
using the fact that $\tilde h$ is anisotropic if and only if
$\tilde h(\tilde x, \tilde x) \ne \tilde 0$ for all nonzero $x\in V$,
as remarked in \cite[p.~101]{RTW}.
\end{proof}

\section{Henselian valuations}
\label{sec:Hensel}

Throughout this section, $(F,v)$ is a Henselian valued field and $A$
is a finite-dimensional simple $F$-algebra with an involution
$\sigma$. We let $K=Z(A)$ and assume $F$ is the subfield of $K$ fixed
by $\sigma$. (Thus, $A$ is central over $F$ if
$\sigma\rvert_{Z(A)}=\id_{Z(A)}$). We assume $A$ is tame over $F$,
which means that $A$ is split by the maximal tamely ramified extension
of $K$, and that $K$ is tame over $F$. Moreover, if
$\charac(\overline{F})=2$ we assume $\sigma$ is not an orthogonal
involution.

\begin{proposition}
\label{existgauges.prop}  With the hypotheses above, every
$v$-gauge on $A$ is tame.
Furthermore, there exist $v$-gauges on $A$ that are invariant under $\sigma$.
\end{proposition}

\begin{proof}
We may represent $A=\End_DV$ for some finite-dimensional right vector
space $V$ over a central division $K$-algebra $D$. Since $v$ is
Henselian, it extends uniquely to a valuation $w$ on $D$ (see for instance
\cite[p.~53, Th.~9]{Sch} or \cite[Th.]{W}). Since $A$ is tame
over $F$, by \cite[Prop.~1.19]{TWgr}
    $D$ must also be
tame over~$F$; hence by \cite[Prop.~1.12 and 1.13]{TWgr} $w$ is a
tame $v$-gauge. Therefore, every $v$-gauge on $A$ is tame, by
\cite[Th.~3.1]{TWgr}.

If $A$ is split and $\sigma$ is symplectic, then
$\sigma=\ad_b$ for
some alternating bilinear form $b$ on $V$,
see \cite[\S4.A]{BoI}. Choose a symplectic base
$\mathcal B = (e_i,f_i)_{i=1}^n$ of $V$ for $b$ and define a $v$-norm
$\alpha$ on
$V$ by
\[
\alpha\bigl(
\tsum_{i=1}^ne_i\lambda_i+f_i\mu_i\bigr) \ = \
\min\limits_{1\leq i\leq n}\bigl(v(\lambda_i),v(\mu_i)\bigr)
\qquad\text{for $\lambda_1$, \ldots, $\mu_n\in F$},
\]
i.e., $\mathcal B$ is a splitting base for $\alpha$
on $V$, and each $\alpha(e_i) = \alpha(f_j) = 0$.
The $v$-norm $\alpha$ on $V$ induces the $v$-gauge
$\End(\alpha)$ on $\End_D(V)$.
For $g\in \End_D(V)$, if $g$ has matrix $(c_{ij})$
relative to $\mathcal B$, then ${\End(\alpha)(g) =
\min_{1\le i,j\le n}v(c_{ij})}$. The matrix for $\sigma(g)$
has the same set of entries up to sign as $(c_{ij})$, though the
entries are relocated.  Hence, $\End(\alpha)$ is invariant
under $\sigma$. We exclude this case of $A$ split
and $\sigma$ symplectic for the rest of the proof. We may
then choose an $F$-linear involution $\theta$ on $D$ of the same type
as $\sigma$ and an even hermitian form $h$ on $V$ with respect to
$\theta$ such that $\sigma=\ad_h$, see \cite[(4.2)]{BoI}. By
\cite[Cor.~3.6]{RTW}, there
exists a $w$-norm $\alpha$ on $V$ that is compatible with $h$. By
\cite[Prop.~1.19]{TWgr}, $\End(\alpha)$ is a $v$-gauge; by
Prop.~\ref{compat.prop}, this gauge is invariant under $\sigma$.
\end{proof}

\begin{theorem}
\label{mainHensel.thm} With the hypotheses of this section,
if $\sigma$ is anisotropic, then  for the Henselian valuation
~$v$
on $F$
there is a unique $\sigma$-special value function $\varphi$ on
$A$ for $v$.
  This~$\varphi$ is a tame $v$-gauge
and its value set $\Gamma_A$ lies in the divisible hull of
$\Gamma_F$. It is the unique $v$-gauge on $A$  invariant under
$\sigma$.
\end{theorem}

\begin{proof}
       We use the same notation as in the proof of
       Prop.~\ref{existgauges.prop}, representing $A=\End_{D}(V)$ as in
       that proof. Since $\sigma$ is anisotropic, it is
       not a symplectic involution on a split algebra.
       Therefore, it is the adjoint
       involution of some even Hermitian form $h$ on $V$ with respect to
       an involution $\theta$ on $D$ of the same type as $\sigma$,
       see \cite[Th.~(4.2)]{BoI}.  The form $h$ is
       anisotropic since $\sigma$ is anisotropic. By
       \cite[Th.~4.6 and Prop.~4.2]{RTW}, the map $\alpha\colon
       V\to\frac12\Gamma_D\cup\{\infty\}$ defined by
       \begin{equation}
       \label{alpha.eq}
       \alpha(x) \, = \, {\textstyle\frac12}w\bigl(h(x,x)\bigr)
       \end{equation}
       is a $w$-norm on $V$ that is compatible with $h$, and the residue form
       $\tilde h$ is anisotropic. Prop.~\ref{compat.prop} then
       shows that $\varphi = \End(\alpha)$ is a surmultiplicative $v$-value
       function on $A$ that is invariant under
       $\sigma$, and Prop.~\ref{adjoint.prop} shows that
$\varphi$ Is $\sigma$-special.
      Since $A$ is
       tame over $F$, the valuation $w$ is a $v$-gauge on $D$ by
       \cite[Prop.~1.13]{TWgr}, hence $\varphi$ is a tame $v$-gauge
       by \cite[Prop.~1.19]{TWgr}. Its value set obviously lies in the
       divisible hull of $\Gamma_D$, which is also the divisible hull of
       $\Gamma_F$.

       To prove uniqueness, suppose $\varphi_{1}$ and $\varphi_{2}$ are each
       $\sigma$-special value functions on $A$ for $v$.
To show that
       $\varphi_{1}=\varphi_{2}$, we argue by induction on the matrix
       size $\ms(A)$, which is defined as the dimension of $V$ in the
       representation $A=\End_D(V)$.

       Suppose first that $A$ is a division algebra. For any subfield
       $L\subseteq A$
       fixed elementwise under $\sigma$ we have
       \[
       \varphi_i(x^2) \, = \, \varphi_i(\sigma(x)x) \, = \, 2\varphi_i(x)
       \qquad\text{for all $x\in L$ and $i=1$, $2$},
       \]
       hence for nonzero
$\xi = \tilde x^{\varphi_i} \in \grof {\varphi_i}L$
we have $\xi^2 = \tilde{x^2}^{\,\varphi_i} \ne 0$.
Therefore,
$\gr_{\varphi_i}(L)$ is semisimple. By
       \cite[Prop.~1.8]{TWgr}, it follows that $\varphi_1$ and
       $\varphi_2$ coincide with the unique valuation on $L$ extending
       $v$. (The extension of $v$ to $L$ is unique because $(F,v)$ is
       Henselian.)
       For any $x\in A$, the product $\sigma(x)x$ lies in a
       subfield of $A$ fixed under $\sigma$, so
       $\varphi_{1}(\sigma(x)x)=\varphi_{2}(\sigma(x)x)$.
       Therefore,
       \[
       \varphi_{1}(x) \, = \, {\textstyle\frac12} 
\varphi_{1}(\sigma(x)x) \, = \,
       {\textstyle\frac12} \varphi_{2}(\sigma(x)x) \, = \,
       \varphi_{2}(x).
       \]
       The claim is thus proved if $\ms(A)=1$.

       Suppose next that $\ms(A)>1$. We may then find in $A$ a symmetric
       idempotent $e\neq0$, $1$. (Representing $A=\End_{D}(V)$ as above,
       we have $\dim_{D}V>1$ and we may take for $e$ the orthogonal
       projection onto any nonzero proper subspace of $V$.) Let $f=1-e$.
       The involution $\sigma$ restricts to $eAe$ and $fAf$, and
       $\ms(eAe)$, $\ms(fAf)<\ms(A)$. By the induction hypothesis, the
       restrictions of $\varphi_{1}$ and $\varphi_{2}$ coincide on $eAe$
       and $fAf$.
       For any $x\in A$, we have $\sigma(xe)xe\in eAe$ and
       $\sigma(xf)xf\in fAf$, hence
       \[
       \varphi_{1}(\sigma(xe)xe) \, = \, \varphi_{2}(\sigma(xe)xe)
       \quad\text{and}\quad
       \varphi_{1}(\sigma(xf)xf) \, = \, \varphi_{2}(\sigma(xf)xf).
       \]
       Since $\varphi_{1}$ and $\varphi_{2}$ are $\sigma$-special 
value functions, Prop.~\ref{eqcond.prop} shows that
       \begin{equation}
	\label{j1j2.eq}
	\varphi_{1}(xe) \, = \, \varphi_{2}(xe)
	\quad\text{and}
	\quad
	\varphi_{1}(xf) \, = \, \varphi_{2}(xf).
       \end{equation}
       On the other hand, we have $xe\sigma(xf)=0$ and $xe+xf=x$, hence
       Prop.~\ref{eqcond.prop} also yields
       \[
       \varphi_{1}(x) \, = \, \min\bigl(\varphi_{1}(xe),\varphi_{1}(xf)\bigr)
       \quad\text{and}\quad
       \varphi_{2}(x) \, = \, \min\bigl(\varphi_{2}(xe),\varphi_{2}(xf)\bigr).
       \]
       By \eqref{j1j2.eq}, it follows that
       $\varphi_{1}(x)=\varphi_{2}(x)$.

       Now, suppose $g$ is a gauge on $A$ that is invariant under
       $\sigma$. By \cite[Th.~3.1]{TWgr} we may find a $w$-norm $\beta$ on
       $V$ such that $g=\End(\beta)$. Up to the addition of a constant,
       we may assume $\beta$ is compatible with~$h$ in view of
       Prop.~\ref{compat.prop}. But the norm $\alpha$ of \eqref{alpha.eq}
       is the only $w$-norm on $V$ that is compatible with $h$ by
       \cite[Prop.~4.2]{RTW}, so $\beta=\alpha$ and $g=\End(\alpha)$.
\end{proof}

If $g$ is a $v$-gauge on $A$ that is invariant under $\sigma$, we
denote by $\sigma_0$ the $0$-component of $\tilde\sigma$. Thus,
$\sigma_0$ is an involution on the $\overline{F}$-algebra
$A_0=A^{\geq0}/A^{>0}$, which may be viewed as the residue algebra of
$A$. The algebra $A_0$ is semisimple, but not necessarily simple, see
\cite[\S2]{TWgr}.
Note that if $A_0 = B_1 \times\ldots\times B_k$ with the $B_i$ simple,
then an involution $\tau$ on $A_0$ is anisotropic if and only if 
$\tau(B_i) = B_i$ and $\tau|_{B_i}$ is anisotropic for $1\le i\le k$.


\begin{corollary}
\label{anisot.cor}
With the hypotheses of this section, if
$g$ is a $v$-gauge on $A$ that is invariant under $\sigma$, the
following conditions are equivalent:
\begin{enumerate}
\item[(a)]
$\sigma$ is anisotropic;
\item[(b)]
$\tilde\sigma$ is anisotropic;
\item[(c)]
$\sigma_0$ is anisotropic.
\end{enumerate}
\end{corollary}

\begin{proof}
The implication (a)~$\Rightarrow$~(b) readily follows from
Th.~\ref{mainHensel.thm}, and the implications (b)~$\Rightarrow$~(a)
and (b)~$\Rightarrow$~(c) are clear. To prove (c)~$\Rightarrow$~(b),
suppose $\xi\in\gr_g(A)$ is a nonzero homogeneous element such that
$\tilde\sigma(\xi)\xi=0$. Every element
$\eta\in\bigl(\xi\gr_g(A)\bigr)\cap A_0$ satisfies
$\tilde\sigma(\eta)\eta=\sigma_0(\eta)\eta=0$. Therefore, $\sigma_0$
is isotropic if $\bigl(\xi\gr_g(A)\bigr)\cap A_0\neq\{0\}$. The
corollary thus follows from the following general result:
\renewcommand{\qed}{\relax}
\end{proof}

\begin{lemma}
       \label{ideal.lem}
       Let $\grA$ be a graded simple algebra
finite-dimensional over a graded field
       $\grK$, and let $\grI\subseteq \grA$ be a homogeneous right  ideal.
       Then, there is a homogeneous idempotent $e\in \grA$ of degree $0$
such that $\grI = e \grA$.
\end{lemma}

\begin{proof}
        By \cite[Prop.~1.3]{HW}, we may identify $\grA=\grEnd_{\grD}(\grV)$
       for some graded  division $\grK$-algebra $\grD$ and some
finite-dimensional graded $\grD$-vector space $\grV$.
Let $\grW = \sum \im(f)$, with the sum taken over all homogeneous $f\in \grI$.
Then, $\grW$ is a graded $\grD$-subspace of $\grV$ and, just as in the
ungraded case, $\grI = \grHom_\grD(\grV,\grW)$.  Take any graded
$\grD$-subspace $\grY$ of $\grV$, such that $\grY$ is complementary to $\grW$,
and let $e\colon \grV\to \grW$ be the projection of $\grV$ onto $\grW$
along $\grY$. Then, the idempotent
$e$ is a degree-preserving graded homomorphism, so $e \in \grA_0$.  Clearly,
$\grI = e\grA$.
\end{proof}

\begin{remarks}
     \label{anisot.rem}
     (1) In Cor.~\ref{anisot.cor}, the hypothesis that $(F,v)$ is
     Henselian is used only to prove that~(a) implies~(b) and~(c); the
     implications (c)$\iff$(b)~$\Rightarrow$~(a) hold without this
     hypothesis (nor any tameness assumption).

     (2) Corollary~\ref{anisot.cor} may be regarded as a version of
     Springer's theorem for involutions. In a slightly different form, it
     has already been proved by Larmour \cite[Th.~4.5]{L}: to see this,
     observe that the residue involutions defined by Larmour are the
     direct summands of our residue involution $\sigma_0$ for a suitable gauge.
\end{remarks}

If the involution $\sigma$ is isotropic, we may still define up to
isomorphism an anisotropic kernel $(A,\sigma)_\an$ in such a way that
if $A=\End_DV$ and $\sigma=\ad_h$, then
$(A,\sigma)_\an\cong(\End_DV_0,\ad_{h_0})$ where $(V_0,h_0)$ is an
anisotropic kernel of $(V,h)$, see \cite{DLT},  and \cite{D} for
involutions of the second kind. The same construction
holds for graded simple algebras with involution.

\begin{theorem}
Let $\sigma_1$, $\sigma_2$ be $F$-linear involutions on $A$ such that
$\sigma_1\rvert_{Z(A)}=\sigma_2\rvert_{Z(A)}$.
For the Henselian valuation $v$ on $F$,
    let $g_1$, $g_2$ be
$v$-gauges on $A$ invariant under
$\sigma_1$ and $\sigma_2$ respectively. If $\charac(\overline{F})=2$,
assume neither $\sigma_1$ nor $\sigma_2$ is orthogonal. The following
conditions are equivalent:
\begin{enumerate}
\item[(a)]
the algebras with involution $(A,\sigma_1)$ and $(A,\sigma_2)$ are
isomorphic;
\item[(b)]
the graded algebras with anisotropic involution
$(\gr_{g_1}(A),\tilde\sigma_1)_\an$
and $(\gr_{g_2}(A),\tilde\sigma_2)_\an$ are isomorphic.
\end{enumerate}
\end{theorem}

\begin{proof}
It follows from Prop.~\ref{type.prop} that $\sigma_i$ and
$\tilde\sigma_i$ are of the same type. Therefore, (a) and (b) each
imply that $\sigma_1$ and $\sigma_2$ are of the same type. If $A$ is
split and $\sigma_1$, $\sigma_2$ are symplectic, then $\gr_{g_1}(A)$
and $\gr_{g_2}(A)$ are split and $\tilde\sigma_1$, $\tilde\sigma_2$
are symplectic, hence hyperbolic. In this case, (a) and (b) both hold
trivially. For the rest of the proof, we exclude this case and fix a
representation $A=\End_DV$ where $V$ is a right vector space over a
central division $K$-algebra $D$. We also fix an involution $\theta$
on $D$ of the same type as $\sigma_1$ and $\sigma_2$, and
non-degenerate even hermitian forms $h_1$, $h_2$ on $V$ with respect
to $\theta$ such that
\[
\sigma_1 \, = \, \ad_{h_1}\qquad\text{and}\qquad\sigma_2 \, = \, \ad_{h_2}.
\]
As observed in the proof of Prop.~\ref{existgauges.prop}, the valuation
$v$ extends uniquely to a valuation $w$ on $D$. By
\cite[Th.~3.1]{TWgr} and Prop.~\ref{compat.prop} we may also find
norms $\alpha_1$ and $\alpha_2$ on $V$ that are compatible with $h_1$
and $h_2$ respectively, such that
\[
g_1 \, = \, \End(\alpha_1)\qquad\text{and}\qquad g_2 \, = \, \End(\alpha_2),
\]
hence
\[
\gr_{g_1}(A) \ = \ \grEnd_{\gr(D)}\bigl(\gr_{\alpha_1}(V)\bigr)
\qquad\text{and}\qquad
\gr_{g_2}(A) \ = \ \grEnd_{\gr(D)}\bigl(\gr_{\alpha_2}(V)\bigr).
\]
It then follows from Prop.~\ref{adjoint.prop} that
\[
\tilde\sigma_1 \, = \, \ad_{\tilde{h_1}}
\qquad\text{and}\qquad
\tilde\sigma_2 \, = \, \ad_{\tilde{h_2}};
\]
hence, denoting by $(\grV_1,k_1)$ and $(\grV_2,k_2)$ the anisotropic kernels of
$(\gr_{\alpha_1}(V),\tilde{(h_1)})$ and
$(\gr_{\alpha_2}(V),\tilde{(h_2)})$ respectively,
\[
(\gr_{g_1}(A),\tilde\sigma_1)_\an \ \cong \ (\grEnd_{\gr(D)}(\grV_1),\ad_{k_1})
\quad\text{ and }\quad
(\gr_{g_2}(A),\tilde\sigma_2)_\an \ \cong \
(\grEnd_{\gr(D)}(\grV_2),\ad_{k_2}).
\]

If (a) holds, then $h_1$ and $h_2$ are similar. Scaling $h_2$ by a
factor in $F^\times$, we may assume $h_1\cong h_2$. By
\cite[Th.~3.11]{RTW}, the anisotropic kernels of
$\tilde{h_1}$ and $\tilde{h_2}$ are
isometric, hence (b) holds.

Conversely, if (b) holds, then the anisotropic kernels of
$\tilde{h_1}$ and $\tilde{h_2}$ are
similar. Scaling $h_2$ by a factor in~$F^\times$, we may assume that
they are isometric. By \cite[Th.~4.6]{RTW}, it follows that $h_1$ and
$h_2$ are isometric, hence (a) holds.
\end{proof}

\begin{corollary}
With the hypotheses of this section,
up to Witt-equivalence the graded algebra with involution
$(\gr_g(A),\tilde\sigma)$ depends only on the Witt-equivalence class
of $(A,\sigma)$, and not on the choice of the invariant $v$-gauge
$g$.
\end{corollary}

\section{Scalar extensions of involutions}
\label{tens.sec}

As an application of the results of \S\ref{sec:Hensel}, we consider a
basic case of the problem of
determining when an anisotropic involution remains anisotropic over a
scalar extension.

Let $\sigma$ be an $F$-linear involution on a finite-dimensional
simple algebra $A$ over a field $F$. Assume $v$~is a
valuation on $F$ and $A$ carries a $v$-gauge $g$ invariant under
$\sigma$. For any extension $(L,v_L)$ of $(F,v)$, we may consider the
$v_L$-gauge $g\otimes v_L$ on the $L$-algebra $A_L=A\otimes_FL$. By
Prop.~\ref{prop:tensprod}, this $v_L$-gauge is invariant under the
involution $\sigma\otimes\id_L$ on $A_L$. If the ``residue''
involution $(\sigma\otimes\id_L)_0$ is anisotropic, then
$\sigma\otimes\id_L$ is anisotropic by Cor.~\ref{anisot.cor} and
Remark~\ref{anisot.rem}(1), and the converse holds if $v_L$ is Henselian
and $A_L$~is tame over $L$, unless $\sigma$ is orthogonal and
$\charac(\overline{F})=2$. We consider below a case where this
residue can be explicitly calculated.

We first recall some  facts which will be used repeatedly below.  Let
$\alpha$ be a surmultiplicative $v$-norm
on a finite-dimensional algebra $A$ over a field $F$ with
valuation $v$.  If $e$ is an idempotent of $A$ with $\al(e) =0$ and
$N$ is any $F$-subspace of $A$, then $\tilde e^{\,2} = \tilde e$ in
$\gr(A)$ and
by \cite[Lemma~1.7]{TWgr},
\begin{equation}\label{idempotent}
\gr(eN) \, = \, \tilde e \gr(N) \quad \text {and} \quad
\gr(Ne)\, = \, \gr(N) \tilde e \quad\text {in} \ \gr(A).
\end{equation}
If $e\ne 1$, let $f = 1-e$.  Then,
$\al(f) \ge\min\big(\al(1),\al(e)\big) = 0$, but since
$f^2 = f$, $\al(f) \le 0$. So, $\al(f) = 0$,  hence
$\tilde f = \tilde 1 - \tilde e$ in $\gr(A)$, and hence
$$
\gr(A) \ = \ \tilde e\gr(A) \oplus \tilde f \gr(A)
\ = \ \gr(eA) \oplus \gr(fA).
$$
Therefore, by \cite[Remark~2.6]{RTW}, the direct sum decomposition
$A = eA \oplus fA$ is a splitting decomposition, i.e.,
$\al(a) = \min\big( \al(ea), \al(fa))$ for any $a\in A$.
Likewise $A = Ae\oplus Af$ is a splitting decomposition.

Recall also that an element $s\in A^\times$ is said to be
{\it$\al$-stable} if $\al(s^{-1})  = -\al(s)$. For
such an $s$ we have by \cite[Lemma~1.3 and (1.5)]{TWgr},
\begin{equation}\label{stable}
   \al(as) \,=\, \al(sa)  \,= \,  \al(a) + \al(s), \quad\text{hence}
\quad \tilde a \,\tilde s \,=\, \tilde {as\,} \quad \text {and} \quad
\tilde s \,\,\tilde a  \, =\, \tilde {sa\,}
\quad\text{
for every $a\in A$}.
\end{equation}

We now make some general observations on the tensor product of
valuations. Let $L/F$ be a finite separable field extension. Recall
that the separability idempotent of $L$ is the idempotent $e\in
L\otimes_FL$  determined uniquely by the conditions that
\begin{equation}
     \label{eq:sepid}
     e\cdot(x\otimes1)\,= \,e\cdot(1\otimes x) \qquad\text{for all $x\in L$}
\end{equation}
and the multiplication map $L\otimes_FL\to L$ carries $e$ to $1$, see
for instance \cite[Prop.~(18.10)]{BoI}. The separability of $L/F$
implies that the
bilinear trace form
\[
T\colon L\times L\to F,\qquad T(x,y)\, =\, \Tr_{L/F}(xy)
\]
is nondegenerate.

\begin{proposition}
     \label{prop:vv}
     Suppose $v\colon F\to\Gamma\cup\{\infty\}$ is a valuation that
     extends uniquely to a valuation $v_L$ on~$L$, and that the valued
     field extension $(L,v_L)$ of $(F,v)$ is tame. Then $v_L$ is a
     $v$-norm on $L$ which is compatible with the bilinear trace form
     $T$, and $(v_L\otimes v_L)(e)=0$.
\end{proposition}

\begin{proof}
     Since $v_L$ is the unique valuation extending $v$ to $L$ and since
     the extension is defectless, it follows that $v_L$ is a $v$-norm
     (indeed, a $v$-gauge) on $L$, see \cite[Cor.~1.9]{TWgr}.

     We claim
     that $v\bigl(\Tr_{L/F}(x)\bigr)\geq v_L(x)$ for all $x\in
     L^\times$. To see this, consider a Galois closure $M$ of $L$ over
     $F$ and an extension $v_M$ of $v$ to $M$. For every $F$-linear
     embedding $\iota\colon L\hookrightarrow M$ the composition
     $v_M\circ\iota$ is a valuation on $L$ extending $v$, hence
     $v_M\circ\iota=v_L$. Since $\Tr_{L/F}(x)=\sum\iota(x)$, where the
     sum extends over all embeddings $\iota\colon L\hookrightarrow M$, we
     have
     \[
     v\bigl(\Tr_{L/F}(x)\bigr) \, = \, v_M\bigl(\tsum\limits_\iota
     \iota(x)\bigr) \,
     \geq \, \min\limits_\iota\bigl(v_M\circ\iota(x)\bigr) \, =  \, v_L(x),
     \]
     proving the claim. It follows that for all $x$, $y\in L^\times$,
     \begin{equation}
       \label{eq:compaT}
       v\bigl(T(x,y)\bigr) \, \geq  \, v_L(x)+v_L(y).
     \end{equation}
     To show that $v_L$ is compatible with $T$, it remains to show that
     for any $x\in L^\times$ there exists $y\in L^\times$ for which
     equality holds in \eqref{eq:compaT}. For this, it suffices to show
     that there exists $\ell\in L^\times$ such that
     $v\bigl(\Tr_{L/F}(\ell)\bigr)=v_L(\ell)$, since equality then  holds
     in~\eqref{eq:compaT} with $y=\ell x^{-1}$. For every $\ell\in
     L^\times$ with $v_L(\ell)=0$ we have
     \begin{equation}
       \label{eq:trace}
       \overline{\Tr_{L/F}(\ell)} \ = \ \IND{\Gamma_L}{\Gamma_F}\cdot
       \Tr_{\overline{L}/\overline{F}}(\overline{\ell})
     \end{equation}
     by \cite[p.~65, Cor.~1]{Er1}. (Ershov assumes his valuation is
Henselian; but the
result carrries over to the situation here: Let
$F_h$ be the Henselization of $F$ with respect to
$v$. Since the unique
extension of $v$ to~$L$ is defectless,
for any compositum of $L$ with $F_h$ we have
$\DIM{L\cdot F_h}{F_h} \ge \DIM{\ov L}{\ov
F}\IND{\Gamma_L}{\Gamma_F}=\DIM LF$.  Hence, $L\otimes_F F_h$ is a
field,
and \eqref{eq:trace} holds for $L/F$ because it holds for
$(L\otimes _F F_h)/F_h$.)

Since $L/F$ is tame, the residue extension
     $\overline{L}/\overline{F}$ is separable and $\charac(\overline{F})$
     does not divide $\IND{\Gamma_L}{\Gamma_F}$. Therefore, we may find
     $\ell\in L$ such that $v_L(\ell)=0$ and
     $\Tr_{\overline{L}/\overline{F}}(\overline{\ell})\neq0$. Then
     \eqref{eq:trace}~ shows that $\overline{\Tr_{L/F}(\ell)}\neq0$, hence
     \[
     v\bigl(\Tr_{L/F}(\ell)\bigr) \, = \, 0 \, = \, v_L(\ell).
     \]
     Therefore, $v_L$ is compatible with $T$; it thus coincides with its
     dual norm $v_L^\sharp$.

     To complete the proof, we compute $(v_L\otimes v_L)(e)$. Let
     $(\ell_i)_{i=1}^n$ be a splitting $F$-base of $L$ for $v_L$, and let
     $(\ell_i^\sharp)_{i=1}^n$ be the dual base for the form $T$. By
     \cite[Prop.~(18.12)]{BoI} we have,
     \[
     e \, = \, \tsum_{i=1}^n\ell_i\otimes\ell_i^\sharp,
     \]
     hence,
     \[
     (v_L\otimes v_L)(e) \ = \ \min\limits_{1\leq i\leq n}\bigl(
     v_L(\ell_i)+v_L(\ell_i^\sharp)\bigr).
     \]
     Now, for all $i=1$, \ldots, $n$ we have
     $v_L(\ell_i^\sharp)=v_L^\sharp(\ell_i^\sharp)=-v_L(\ell_i)$ by
     \cite[Lemma~3.4]{RTW}. Therefore, $(v_L\otimes v_L)(e)=0$.
\end{proof}

Continuing with the same notation and hypotheses as in
Prop.~\ref{prop:vv}, we now assume further that the extension $L/F$ is
Galois. Let $G$ denote its Galois group.
Since $v_L$ is the unique extension of $v$ to $L$,
$v_L \circ \iota = v_L$ for any $\iota \in G$, and hence $\iota$
induces a graded $\gr(F)$-automorphism $\tilde\iota$ of $\gr(L)$.
For $\iota\in G$, let
\[
e_\iota=(\id\otimes\iota)(e)\in L\otimes_FL,
\]
and let $\tilde e_\iota$ be the image of $e_\iota$ in
$\gr(L\otimes_FL)$, which is canonically  identified with $\gr(L) \otimes
_{\gr(F)} \gr(L)$ by Prop.~\ref{prop:tensprod}.

\begin{lemma}
     \label{lem:vv}
     The elements $(e_\iota)_{\iota\in G}$ form a family of orthogonal
     idempotents such that $\sum_{\iota\in G}e_\iota=1$.
They are the primitive idempotents of $L\otimes_F L$.
They satisfy
     $(v_L\otimes v_L)(e_\iota)=0$ and
     \begin{equation}
       \label{eq:sepidi}
       e_\iota\cdot(x\otimes1) \, = \, e_\iota\cdot\bigl(1\otimes\iota(x)\bigr)
       \qquad\text{for $x\in L$}.
     \end{equation}
Likewise,
for any $y\in \gr(L)$,
\begin{equation}
       \label{eq:tildesepidi}
       \tilde e_\iota\cdot(y\otimes \tilde 1)\, =\,
\tilde e_\iota\cdot\bigl(\tilde 1\otimes\tilde\iota(y)\bigr)
       \quad\text{in \ $\gr(L) \otimes_{\gr(F)}\gr(L)$}.
     \end{equation}
Moreover, $(\iota\otimes\iota)(e)=e$ for $\iota\in G$.
\end{lemma}

\begin{proof} Equation \eqref{eq:sepid} shows that $e \cdot(L\otimes_F L)
= e\cdot (L\otimes 1) \cong L$.  Since $L$ is a field,
$e$ must be a primitive idempotent.
     Equation~\eqref{eq:sepidi} readily follows by applying
     $\id\otimes\iota$ to each side of \eqref{eq:sepid}.
    For equation~\eqref{eq:tildesepidi}, it suffices to verify the
equality when  $y$ is
homogeneous and nonzero. But then $y = \tilde x$ for
some nonzero~$x\in L$.
Both $x\otimes 1$
and $1\otimes \iota(x)$ are $v_L\otimes v_L$-stable in
$L\otimes _F L$, as defined preceding \eqref{stable}~above.
    Hence, using equations \eqref{stable} and
\eqref{eq:sepidi},
$$
\tilde e_\iota \cdot \big(\tilde x \otimes \tilde1\big) \, = \, \tilde e_\iota
\cdot \big(\tilde{x\otimes 1}\big) \, = \, \big[e_\iota\cdot(x\otimes 1)\big]
\,\,\tilde
\, \,= \, \big[e_\iota\cdot\bigl(1\otimes
\iota(x)\bigr)\big]\,\,\tilde  \, \,= \,
\tilde e_\iota\cdot\big(\tilde{1\otimes \iota(x)}\big) \, = \,
\tilde e_\iota\cdot \big(\tilde 1 \otimes \tilde \iota(\tilde x)\big).
$$
Since $e$ is a primitive
     idempotent, it is clear that each $e_\iota$ is
also a primitive idempotent. For
     $\iota$, $\kappa\in G$ and $x\in L$, as $L\otimes_F L$~is commutative
   we have
$$
e_\iota e_\kappa\cdot\bigl(1\otimes[\kappa(x)-\iota(x)]\bigr)
   \ = \ e_\iota e_\kappa\cdot\bigl(1\otimes\kappa(x)\bigr)
   \, -\, \ e_\kappa e_\iota \cdot\bigl(1\otimes\iota(x)\bigr)
   \ = \ \big(e_\iota e_\kappa- e_\kappa e_\iota\big)\cdot (x\otimes 1)
   \ = \ 0.
$$
For $\iota \ne \kappa$, if we choose $x\in L$ with $\iota(x) \ne
\kappa(x)$, then $1\otimes[\kappa(x)-\iota(x)]$ is a unit of
$L\otimes_F L$; hence, $e_\iota e_\kappa=0$.

     As observed in the proof of Prop.~\ref{prop:vv}, we have
     $e=\sum_{i=1}^n\ell_i\otimes\ell_i^\sharp$ if $(\ell_i)_{i=1}^n$ is
     an $F$-base of $L$ and $(\ell_i^\sharp)_{i=1}^n$ is the dual base
     for the bilinear form $T$. It follows that
     $e_\iota=\sum_{i=1}^n\ell_i\otimes\iota(\ell_i^\sharp)$ for
     $\iota\in G$, hence
     \begin{equation}
       \label{eq:sumid}
       \tsum_{\iota\in G}e_\iota \ = \ \tsum_{i=1}^n\ell_i\otimes
       \Tr_{L/F}(\ell_i^\sharp).
     \end{equation}
     Since $(\ell_i^\sharp)_{i=1}^n$ is the dual base of
     $(\ell_i)_{i=1}^n$, we have
     \[
     x \ = \ \tsum_{i=1}^n\ell_i\Tr_{L/F}(\ell_i^\sharp x) \qquad\text{for
       $x\in L$}.
     \]
     In particular, $\sum_{i=1}^n\ell_i\Tr_{L/F}(\ell_i^\sharp)=1$, and
     equation~\eqref{eq:sumid} yields $\sum_{\iota\in G}e_\iota=1$.
So, the $e_\iota$ are all the primitive idempotents of $L\otimes_F L$.

     Since $v_L$ is the unique valuation extending $v$ to $L$, we have
     $v_L\circ\iota=v_L$ for all $\iota\in G$, hence
     \[
     (v_L\otimes v_L)(e_\iota)=(v_L\otimes v_L)(e)=0 \qquad\text{for all
       $\iota\in G$}.
     \]
     Finally, it is clear that $(\iota\otimes\iota)(e)$ satisfies the
     same equation~\eqref{eq:sepid} as $e$ and is carried to $1$ by the
     multiplication map $L\otimes_FL\to L$. Since these properties
     determine $e$ uniquely, we have $(\iota\otimes\iota)(e)=e$ for all
     $\iota\in G$.
\end{proof}

It is well-known (cf.~ \cite[pp.~256--257, Lemma b]{P}) that the
primitive idempotents of $L\otimes_F L$
are indexed by $G$ and satisfy \eqref{eq:sepidi}.
The further properties of the $e_\iota$ given
in Lemma~\ref{lem:vv} will
be useful in what follows.

Now, assume further that $L\subseteq D$ for some finite-dimensional
division $F$-algebra $D$, and that $v$ extends to a valuation $v_D$ on
$D$ such that $D/F$ is defectless; i.e., $v_D$ is a $v$-norm on
$D$. The restriction of $v_D$ to $L$ is then the unique valuation
$v_L$ extending $v$.  We will use the idempotents $(e_\iota)_{\iota\in G}$
to analyze extensions of involutions from $D$ to $D\otimes_F L$.
Let $C$ be the centralizer $C_D(L)$.  Viewing $D$ as a right
$C$-vector space, we have the canonical isomorphism
\begin{equation}
    \label{eq:canisom}
    \eta\colon D\otimes_FL \ \iso \ \End_C(D),
\end{equation}
which carries $d\otimes\ell$ to the map $x\mapsto dx\ell$ for $d$,
$x\in D$ and $\ell\in L$. For $\iota\in G$, consider the following
$C$-subspace of $D$:
$$
D_\iota\ = \ \{ d\in D \,| \ \ell d = d\,\iota(\ell) \text{ for all }
\ell \in L\}.
$$
Since $\iota$ on $L$ is induced by an inner automorphism of $D$
by Skolem-Noether, $D_\iota\ne \{0\}$. Since in addition,
$D_{\id} = C$ and $D_\kappa \cdot D_\iota \subseteq D_{\iota \kappa}$
for all $\iota, \kappa\in G$, we must have $\dim_C(D_\iota) = 1$
for each $\iota$.

\enlargethispage{\baselineskip}
\enlargethispage{\baselineskip}

\begin{lemma}
     \label{prop:dirsum}
     We have $D\otimes_FL=\bigoplus_{\iota\in G}
     e_\iota\,(D\otimes1)$ and $D=\bigoplus_{\iota\in G}D_\iota$, and
     these direct sums are splitting decompositions of $D\otimes_FL$ and
     $D$ with respect to $v_D\otimes v_L$ and $v_D$, respectively. More
     precisely, we have
     \[
     (v_D\otimes v_L)\bigl(\tsum_{\iota\in G}
     e_\iota\cdot(x_\iota\otimes1) \bigr) = \min\limits_{\iota\in
       G}\bigl(v_D(x_\iota)\bigr) \qquad\text{for $x_\iota\in D$},
     \]
     and
     \[
     v_D\bigl(\tsum_{\iota\in G}y_\iota\bigr)=\min\limits_{\iota\in G}\bigl(
     v_D(y_\iota)\bigr)\qquad\text{for $y_\iota\in D_\iota$}.
     \]
     Furthermore, for all $\iota$, $\kappa\in G$,
     \[
     e_\iota (D\otimes_F L) e_\kappa =  e_\iota (D_{\kappa^{-1}\iota}
     \otimes 1).
     \]
\end{lemma}

\begin{proof}
Let $A = D\otimes_F L$ and $\al = v_D\otimes v_L$.
Since $(e_\iota)_{\iota\in G}$ is a family of  orthogonal idempotents
     with $\sum_{\iota\in G}e_\iota=1$ and $\al(e_\iota) = 0$
for each $\iota$, the collection  $(\tilde e_\iota)_{\iota \in G}$
is a family of orthogonal idempotents in $\gr(A)$ with
$\sum_{\iota\in G}\tilde e_\iota=\tilde 1$. Hence,
using \eqref{idempotent},
\begin{equation}\label{direct sum}
A \, = \,  \textstyle \bigoplus\limits_{\iota \in G}e_\iota A
\quad\text{and}\quad \gr(A)  \, = \,
\bigoplus\limits_{\iota\in G} \tilde e_\iota  \gr(A) \, = \,
\bigoplus\limits_{\iota\in G}\gr(e_\iota A).
\end{equation}
Likewise, for any $\iota\in G$,
\begin{equation}\label{iotadirect sum}
e_\iota A \, = \,  \textstyle \bigoplus\limits_{\kappa \in G}e_\iota A e_\kappa
\quad\text{and}\quad \gr(e_\iota A)  \, = \,
\bigoplus\limits_{\kappa\in G} \gr(e_\iota A)\tilde e_\kappa \, = \,
\bigoplus\limits_{\kappa\in G}\gr(e_\iota A e_\kappa).
\end{equation}
In view of \eqref{eq:sepidi}, we have $e_\iota\cdot(1\otimes
     L)=e_\iota\cdot(L\otimes1)$, hence $e_\iota\cdot(D\otimes_FL)=
     e_\iota\cdot(D\otimes1)$.
For any nonzero $x\in D$,  $\al(x\otimes 1) = v_D(x)$ and
$\al\big((x\otimes 1)^{-1}\big) = \al(x^{-1}\otimes 1) = -v_D(x)$.
So, $x\otimes 1$ is $\al$-stable, and
\eqref{stable}~applies.  Since \eqref{direct sum}
shows that the direct sum $A =\bigoplus_{\iota \in G}e_\iota A$ is a
splitting decomposition of $A$ for~$\al$, it follows using \eqref{stable}
that for any $x_\iota\in D$,
$$
\al\big(\tsum\limits_{\iota \in G}e_\iota\cdot (x_\iota\otimes 1)\big)
   \ = \ \min\limits_{\iota\in G}\big(\al(e_\iota\cdot (x_\iota\otimes 1))\big)
\ = \ \min\limits_{\iota \in G}\big(\al(e_\iota) + \al(x_\iota \otimes 1)\big)
   \ = \ \min\limits_{\iota \in G}\big(v_D(x_\iota)\big).
$$

To prove the rest, we use the canonical isomorphism $\eta$ of
\eqref{eq:canisom}. For each $\iota \in G$, let $\pi_\iota = \eta(e_\iota)$,
which is a projection in $\End_C(D)$.
By \eqref{eq:sepidi} and the commutativity of $L\otimes_FL$,
for any $\ell \in L$ and $d\in D$,
$$
\ell \cdot\pi_\iota(d) \ = \
\eta\big((\ell \otimes 1) e_\iota\big)(d) \
= \ \eta\big((1\otimes \iota(\ell))e_\iota\big)(d)   \ = \
\pi_\iota(d) \cdot \iota(\ell) .
$$
Hence, $\pi_\iota(D)\subseteq D_\iota$.  Since
$\im(\pi_\iota)$ is a nonzero $C$-subspace of the
$1$-dimensional $C$-vector space $D_\iota$, it follows that
$\im(\pi_\iota) = D_\iota$.
Because $(\pi_\iota)_{\iota\in
    G}$ is a family of orthogonal idempotents of $\End_C(D)$ such that
$\sum_{\iota\in G}\pi_\iota=\id_D$, we have
$D=\bigoplus_{\iota\in G}D_\iota$;
furthermore,  each $\pi_\iota$ is the
projection of $D$ onto $D_\iota$ parallel to
$\bigoplus_{\kappa\neq\iota}D_\kappa$.
Thus, for any $\iota, \kappa \in G$,
$\pi_\iota\End_C(D)\pi_\kappa$
consists of those
endomorphisms sending $D_\kappa$ to $D_\iota$ and $D_\lambda$ to $\{0\}$
for $\lambda \ne \kappa$.
  For any $\lambda \in G$, since
$D_{\kappa^{-1}\iota} D_\lambda \subseteq
D_{\lambda\kappa^{-1}\iota}$, we have
$$
[\pi_\iota \circ\eta(D_{\kappa^{-1}\iota}\otimes 1)](D_\lambda)
  \ = \ \pi_\iota(D_{\kappa^{-1}\iota} D_\lambda) \
\subseteq  \ \pi_\iota(D_{\lambda\kappa^{-1}\iota}) \
\subseteq  \ \begin{cases}  \ D_\iota, &\text{if } \lambda = \kappa;\\
  \ \{0\}, &\text{if }\lambda \ne \kappa.
\end{cases}
$$
Hence, $\pi_\iota \circ\eta(D_{\kappa^{-1}\iota}\otimes 1)
\subseteq \pi_\iota\End_C(D)\pi_\kappa$.
By applying $\eta^{-1}$, this yields
\begin{equation}\label{inclusion}
e_\iota(D_{\kappa^{-1}\iota}\otimes 1) \,\subseteq  \,
e_\iota A e_\kappa \quad\text {for all } \iota, \kappa \in G.
\end{equation}

Now, fix $\iota \in G$.  We have seen that $e_\iota A = 
e_\iota(D\otimes 1)$.  The $F$-epimorphism $\rho_\iota\colon D \to 
e_\iota A$ given by $d\mapsto e_\iota(d\otimes 1)$ is clearly
injective; $\rho_\iota$ is also norm-preserving, as $\alpha(e_\iota) = 0$
and $d\otimes 1$ is stable in $A$ for each nonzero $d\in D$.
Since $D = \bigoplus_{\kappa \in G}D_{\kappa^{-1}\iota}$, we have
\begin{equation}\label{eAe}
\textstyle \bigoplus \limits_{\kappa \in G} e_\iota  A e_\kappa
  \ = \ e_\iota A \ = \ \rho_\iota(D) \ = \
\bigoplus\limits _{\kappa \in G}\rho_\iota(D_{\kappa^{-1}\iota}) \ = 
\ \bigoplus\limits _{\kappa \in G}
e_\iota(D_{\kappa^{-1}\iota}\otimes 1).
\end{equation}
This shows that the inclusions in \eqref{inclusion} must
all be equalities.  It follows from \eqref{iotadirect sum} above that 
the direct sum decomposition
$\bigoplus _{\kappa \in G} e_\iota  A e_\kappa$ is a splitting
decomposition of $e_\iota A$.  Therefore, by applying the 
norm-preserving map $\rho_\iota^{-1}$
to the terms in \eqref{eAe}, it follows that
$\bigoplus_{\kappa \in G} D_\kappa$ is a splitting
decomposition of~$D$.
\end{proof}

While $D\otimes_F L$ is simple, the degree $0$ part
$(D\otimes_F L)_0$ of $\gr(D\otimes_F L)$ is in general only
semisimple.  The value sets $\Gamma_{D_\iota}$ of the $D_\iota$
encode how $(D\otimes_F L)_0$ decomposes:  Since each $D_\iota$ is a
$1$-dimensional $C$-subspace of $D$ and $v_D|_{D_\iota}$
is a $v_D|_C$-norm on $D_\iota$, each $\Gamma_{D_\iota}$ is a
coset of $\Gamma_C$ in $\Gamma_D$.
Therefore, there is a well-defined map
$$
\psi\!\colon G \to \Gamma_D\big/\Gamma_C
\quad \text{given by} \quad \psi(\iota) \, =\,  \Gamma_{D_\iota}.
$$
Because $D_\iota\cdot D_\kappa \subseteq D_{\kappa \iota}$
and $\Gamma_D$ is abelian, $\psi$ is a group homomorphism,
which is surjective since ${D = \bigoplus_{\iota\in G} D_\iota}$
is a splitting decomposition of $D$ by Lemma~\ref{prop:dirsum}.
So, ${|\ker(\psi)| = |G|\big/ \IND {\Gamma_D}{\Gamma_C}
= \DIM DC\big/\IND {\Gamma_D}{\Gamma_C}}$, which shows that
\begin{equation}\label{injcond}
\psi\text{ is injective if and only if $D$ is totally
ramified over $C$.}
\end{equation}

\begin{lemma}\label{etildeinfo}
Let $A = D \otimes _F L$.  Then,
each $\tilde e_\iota$ is a primitive idempotent of $A_0$,
and \break $A_0 = \bigoplus\limits _{\iota \in G}\tilde e_\iota (A_0)
=\bigoplus\limits_{\iota \in G}\tilde e_\iota(D_0\otimes 1)$.
For any $\iota, \kappa \in G$,
$$
\tilde e_\iota A_0\tilde e_\kappa \ = \ \begin{cases}\
\tilde e_\iota\big((D_{\kappa^{-1}\iota})_0 \otimes 1\big) \ne 0,
& \text{if } \psi(\kappa ) = \psi(\iota);\\
   \  \ \ \ \ \ \ \ \ \  \ 0 ,&\text{if } \psi(\kappa) \ne \psi(\iota).
\end{cases}
$$
\end{lemma}

\begin{proof}
We saw in \eqref{direct sum} that $\gr(A) = \bigoplus_{\iota \in
G}\tilde e_\iota \gr(A)$.  Moreover, as
$\gr(A) = \gr(D) \otimes_{\gr(F)}\gr(L)$ and
$\tilde e_\iota\big(1\otimes \gr(L)\big)=
\tilde e_\iota\big(\gr(L)
\otimes 1\big)$ by \eqref{eq:tildesepidi}, we have
$\tilde e_\iota \gr(A) =
\tilde e_\iota\big(\gr(D) \otimes 1 \big)$.  So, for the
degree $0$ components we have $A_0 = \bigoplus_{\iota\in G}
\tilde e_\iota (A_0) = \bigoplus_{\iota \in G}
\tilde e_\iota\big(D_0\otimes 1\big)$.  Similarly,
for $\iota, \kappa\in G$, by \eqref{idempotent} and
Lemma~\ref{prop:dirsum}
$$
\tilde e_\iota \gr(A) \tilde e_\kappa\ = \
\gr(e_\iota A e_\kappa) \ = \ \gr\big(e_\iota
(D_{\kappa^{-1}\iota}\otimes 1)\big)  \
= \ \tilde e_\iota\gr(D_{\kappa^{-1}\iota}\otimes 1)
   \ = \ \tilde e_\iota\big(\gr(D_{\kappa^{-1}\iota}) \otimes 1\big).
$$
Hence, for the degree $0$ components,
$$
\tilde e_\iota (A_0)\tilde e_\kappa \ = \ \tilde e_\iota
\big((D_{\kappa^{-1}\iota})_0 \otimes 1\big).
$$
If $\psi(\kappa) \ne \psi(\iota)$, then
$\psi(\kappa^{-1}\iota)$ is a nonzero element of
$\Gamma_D\big/\Gamma_C$, so $(D_{\kappa^{-1}\iota})_0 = \{0\}$.  If
$\psi(\kappa) = \psi(\iota)$, then $(D_{\kappa^{-1}\iota})_0
\ne 0$, and since nonzero elements of $D_{\kappa^{-1}\iota}
\otimes 1$ are stable, \eqref{stable} yields
$\tilde e_\iota \big((D_{\kappa^{-1}\iota})_0\otimes 1\big)
\ne \{0\}$.  If $\kappa = \iota$, then $(D_{\kappa^{-1}\iota})_0
\otimes 1 = C_0\otimes 1$, so
$\tilde e_\iota (A_0) \tilde e_\iota = \tilde e_\iota
\big(C_0\otimes 1\big)$.   Since $C\otimes 1$ centralizes
$e_\iota \in L\otimes _F L$, $C_0\otimes 1$~centralizes
$\tilde e_\iota$.  Hence, $\tilde e_\iota(A_0)\tilde e_\iota
\cong C_0\otimes 1 \cong C_0$. Since $C_0$ is a division ring,
$\tilde e_\iota$ is a primitive idempotent of $A_0$.
\end{proof}

Now, assume $\sigma$ is an $F$-linear involution on $D$ which
stabilizes $L$, and therefore restricts to an automorphism $\sigma_L$
of $L$, and let $\iota\in G$ be such that $\iota^2=\id$. Then
$\sigma\otimes\iota$ is an involution on $D\otimes_FL$. Since the
valuation $v_D$ extending $v$ to $D$ is unique by
\cite[Th.]{W}, it is
invariant under $\sigma$. Likewise, $v_L$~is invariant under $\iota$,
hence $v_D\otimes v_L$ is invariant under $\sigma\otimes\iota$ by
Prop.~\ref{prop:tensprod}.

\begin{proposition}
     \label{prop:isotcrit}
     The involution $\sigma\otimes\iota$ on $D\otimes_FL$ is isotropic
     unless $\sigma_L=\iota$ and $\iota$ lies in the center
$Z(G)$ of $G$. If
     $\sigma_L=\iota\in Z(G)$ and
$D$ is totally ramified over $C_D(L)$,
then   $\sigma\otimes\iota$ is anisotropic.
\end{proposition}

\begin{proof}
     For $\kappa\in G$ we have
     \[
     (\sigma_L\otimes\iota)(e_\kappa)=(\sigma_L\otimes\iota\kappa)(e)
     =(\id_L\otimes\iota\kappa\sigma_L)\circ(\sigma_L\otimes\sigma_L)(e).
     \]
     Since $(\sigma_L\otimes\sigma_L)(e)=e$ by Lemma~\ref{lem:vv}, it
     follows that
     \[
     (\sigma_L\otimes\iota)(e_\kappa)=e_{\iota\kappa\sigma_L}.
     \]
     If $\iota\neq\sigma_L$ or if $\iota=\sigma_L$
and $\iota\notin Z(G)$, we may find $\kappa\in G$ such that
     $\iota\kappa\sigma_L\neq\kappa$, hence
     \[
     (\sigma\otimes\iota)(e_\kappa)\cdot
     e_\kappa=e_{\iota\kappa\sigma_L}\cdot e_\kappa=0.
     \]
     Therefore, $\sigma\otimes\iota$ is isotropic.

Now assume $\sigma_L=\iota$ and $\iota\in Z(G)$.
     So, $(\sigma\otimes\iota)(e_\kappa)=e_\kappa$ for all $\kappa\in
G$;  hence, in $(D\otimes_F L)_0$,
${(\sigma\otimes \iota)_0(\tilde e_\kappa) = \tilde e_\kappa}$.
Assume further that $D$ is totally ramified over $C = C_D(L)$.
Then $\psi$ is injective by~\eqref{injcond}, so by
Lemma~\ref{etildeinfo}, $\tilde e_\iota(D\otimes _F L)_0\,\tilde e_\kappa
= 0 $ whenever $\kappa \ne \iota$.  Hence,
${(D\otimes _F L)_0 = \bigoplus_{\iota \in G}
\tilde e_\iota(D\otimes _F L)_0\,\tilde e_\iota}$.  Since
$(\sigma\otimes \iota)_0$ maps each direct summand to itself
and each summand is a division ring, $(\sigma\otimes \iota)_0$
is anisotropic.  It follows from
     Cor.~\ref{anisot.cor} (see also Remark~\ref{anisot.rem}(1)) that
     $\sigma\otimes\iota$ is anisotropic.
\end{proof}

\begin{corollary}
     \label{cor:isotcrit}
     Let $D$ be a central division algebra over a field $F$. Assume $v$
     is a valuation on $F$ which extends to a valuation on $D$ so that
     $D$ is tame over $F$. Let $\sigma$ be an involution of
     the first kind on $D$ and let $L\subseteq D$ be a subfield
     Galois over $F$, consisting of $\sigma$-symmetric elements.
If $D$ is totally ramified over $C_D(L)$, then the
     involution $\sigma\otimes\id_L$ on $D\otimes_FL$ is anisotropic.
\end{corollary}

\begin{proof}
This is immediate from Prop.~\ref{prop:isotcrit}.
\end{proof}

\begin{remarks}
(a) The assumption in  Cor.~\ref{cor:isotcrit} that
$D$~is totally ramified over $C_D(L)$ holds whenever
$D$~is totally ramified over $F$.  In this case we do not have to
assume that $L$ is Galois over $F$.  For,
since $v$ extends to $D$, it follows from a theorem of Morandi
     \cite{M} that $D$ remains a division ring after scalar extension to a
     Henselization $F_h$ of $F$ for $v$. Therefore, we may assume that $F$ is
     Henselian. The extension $L/F$ is then Galois, since it is tame and
     totally ramified.

(b) Another case in which $D$ is totally ramified over $C_D(L)$
occurs whenever the subfield $L$ of $D$ is unramified over $F$ and
$\ov L \subseteq Z(\ov D)$.

(c) Another way to obtain the information
about $(D\otimes _FL)_0$ needed in the proof of
Prop.~\ref{prop:isotcrit} is to prove that if
the $F$-central division ring $D$ has a valuation
tame over $F$ and $L$ is any subfield of $D$
containing $F$, and $C = C_D(L)$,
then the canonical isomorphism $D\otimes _F L \cong
\End_C(D)$ is norm-preserving; so this induces
a graded isomorphism $\gr(D\otimes _F L)
\cong \gr\big(\End_C(D)\big) \cong \grEnd
_{\gr(C)}\big(\gr(D)\big)$.
\end{remarks}

\section{Composition of value functions}
\label{comp.sec}

Let $v\colon F\to\Gamma\cup\{\infty\}$ be a valuation on a field $F$,
and let $\Delta\subset\Gamma$ be a convex subgroup, i.e., if
$0\le \gamma\le \delta $ with $\gamma\in \Gamma$ and $\delta \in \Delta$,
then $\gamma\in \Delta$.  Let $\Lambda = \Gamma/\Delta$, and let
$\varepsilon\colon \Gamma \to \Lambda$ be the canonical map.
The ordering on $\Gamma$ induces a
total ordering on $\Lambda$ such that for $\gamma_1, \gamma_2\in \Gamma$, if
$\gamma_1 \le \gamma_2$, then $\varepsilon(\gamma_1) \le
\varepsilon(\gamma_2)$.
Consequently,
\begin{equation}\label{ordinequality}
\text {if }  \ \varepsilon(\gamma_2) \, < \, \varepsilon(\gamma_1),
\quad \text { then }  \ \gamma_2  \, < \,  \gamma_1.
\end{equation}
Because $\Gamma$ is assumed to be divisible, $\Delta$ and $\Lambda$
are also divisible.
     By composing $v$ with $\varepsilon$, we obtain a coarser valuation on
$F$,
\[
w\, =\, \varepsilon\circ v\colon F\to\Lambda\cup\{\infty\}.
\]
Let $\overline{F}^{v}$ (resp.\ $\overline{F}^w$) denote the residue
field of $F$ for the valuation $v$ (resp.\ $w$). The valuation $v$
induces a valuation
\[
u\colon\overline{F}^w\to\Delta\cup\{\infty\},
\]
with residue field
\[
\overline{\overline{F}^w}^u =\, \overline{F}^v,
\]
see \cite[pp.~44--45]{EP}.

Now, let $V$ be an $F$-vector space and let $\alpha\colon
V\to\Gamma\cup\{\infty\}$ be a $v$-value function. Composition with~
$\varepsilon$ yields a
$w$-value function
\[
\beta\, = \, \varepsilon\circ\alpha\colon V\to\Lambda\cup\{\infty\}.
\]
Each $\lambda\in\Lambda=\Gamma/\Delta$ is a coset of $\Delta$, and
may therefore be viewed as a subset of $\Gamma$. For $x\in V$, we have
by definition
\[
\beta(x)\, = \, \lambda\in\Lambda\quad\text{if and only
     if}\quad\alpha(x)\in\lambda\subset\Gamma.
\]
For $\lambda\in\Lambda$, let
\[
V^{\beta\geq\lambda} \, = \, \{x\in V\mid \beta(x)\geq\lambda\},\quad
V^{\beta>\lambda} \, = \, \{x\in V\mid \beta(x)>\lambda\},\quad
V^\beta_\lambda \, = \, V^{\beta\geq\lambda}/V^{\beta>\lambda}.
\]
The group $V^\beta_\lambda$ is an $\overline{F}^w$-vector space.

\begin{lemma}
     If $x$, $y\in V^{\beta\geq\lambda}$ satisfy $x\equiv
     y\not\equiv0\bmod V^{\beta>\lambda}$, then $\alpha(x)=\alpha(y)$.
\end{lemma}

\begin{proof}
We have $\beta(x-y) >\lambda = \beta(y)$.  Since
$\beta = \varepsilon \circ \alpha$,
\eqref{ordinequality} shows that $\alpha(x-y) > \alpha(y)$.
Hence, $\alpha(x) = \min\big(\alpha(x-y), \alpha(y)\big) = \alpha(y)$.
\end{proof}

In view of this lemma, we may define
\[
\alpha_{\lambda}\colon
V^{\beta}_\lambda\to\lambda\cup\{\infty\}\quad\text{by}
\quad
x+V^{\beta>\lambda}\mapsto
\begin{cases}
       \alpha(x)&\text{if $\beta(x)=\lambda$},\\
       \infty&\text{if $\beta(x)>\lambda$}.
\end{cases}
\]
Clearly, $\alpha_{\lambda}$ is a $u$-value function on $V^\beta_{\lambda}$.
For $\gamma\in\lambda$ we have
\[
(V^\beta_{\lambda})^{\alpha_\lambda}_{\gamma} \, = \, V^\alpha_{\gamma}.
\]
Therefore,
\[
\gr_{\alpha}(V)\ = \ \tbigoplus_{\lambda\in\Lambda}
\gr_{\alpha_{\lambda}}(V^\beta_{\lambda})
\quad\text{where}\quad
\gr_{\alpha_{\lambda}}(V^\beta_{\lambda})\ = \  \tbigoplus_{\gamma \in \lambda}
V_\gamma^\alpha
\qquad\text{while}\qquad
\gr_{\beta}(V) \ = \ \tbigoplus_{\lambda\in\Lambda}V^\beta_{\lambda}.
\]
Now, let
\[
\Gamma_{F} \, = \, v(F^\times) \, \subseteq \, \Gamma,\qquad
\Delta_{F} \, = \, \Delta\cap\Gamma_{F} \, \subseteq \, \Delta,\qquad
\Lambda_{F} \, = \, w(F^\times) \, = \, \Gamma_{F}/\Delta_{F} \,
\subseteq \, \Lambda.
\]
These groups are the value groups of, respectively, $v$, $u$, and $w$.
Similarly, let
\[
\Gamma_{V} \, = \, \alpha(V\setminus\{0\}) \, \subseteq \, \Gamma
\quad\text{and}\quad
\Lambda_{V} \, = \, \beta(V\setminus\{0\}) \, \subseteq \, \Lambda.
\]
For each $\lambda\in\Lambda_V$, let also
\[
\lambda_V \, = \, \alpha_\lambda(V_\lambda^\beta\setminus\{0\})
    \, \subseteq \, \lambda.
\]
Clearly, $\lambda_V=\lambda\cap\Gamma_V$.
Note $\Gamma_{F}$ (resp.\ $\Delta_{F}$, resp.\ $\Lambda_{F}$) is a
subgroup of $\Gamma$ (resp.\ $\Delta$, resp.\ $\Lambda$), while
$\Gamma_{V}$ (resp.\ $\Lambda_{V}$, resp.\ $\lambda_V$ for
$\lambda\in\Lambda_V$) is a union of cosets of
$\Gamma_{F}$ (resp.\ $\Lambda_{F}$, resp.\ $\Delta_F$). We denote by
$\IND{\Gamma_V}{\Gamma_F}$ the cardinality of the set of cosets of
$\Gamma_F$ in $\Gamma_V$, and define likewise
$\IND{\Lambda_V}{\Lambda_F}$ and $\IND{\lambda_V}{\Delta_F}$ for
$\lambda\in\Lambda_V$.

\begin{lemma}
       If $\dim_{F}V$ is finite, then $\IND{\Gamma_{V}}{\Gamma_{F}}$,
       $\IND{\Lambda_{V}}{\Lambda_{F}}$, and $\IND{\lambda_V}{\Delta_F}$
       for $\lambda\in\Lambda_V$ are finite. If $\lambda_{1}$,
       \ldots, $\lambda_{r}\in\Lambda_{V}$ are representatives of the
       various cosets of $\Lambda_{V}$ modulo $\Lambda_{F}$, then
       \[
       \IND{\Gamma_{V}}{\Gamma_{F}} \ = \
       \tsum_{i=1}^r\IND{(\lambda_{i})_V}{\Delta_{F}}.
       \]
\end{lemma}

\begin{proof}
       By \cite[Prop.~2.2]{RTW} we have
       \[
       \IND{\Gamma_{V}}{\Gamma_{F}} \, \leq \, \dim_{F}V,\qquad
       \IND{\Lambda_{V}}{\Lambda_{F}} \, \leq \, \dim_{F}V,
       \]
       and also
       \[
       \IND{\lambda_V}{\Delta_F} \, \leq \, \dim_{\overline{F}^w}V_{\lambda}
        \, \leq \, \dim_FV \qquad\text{for $\lambda\in\Lambda_V$}.
       \]
       For $i=1$, \ldots, $r$, let $\gamma_{i1}$, \ldots,
       $\gamma_{is_{i}}\in(\lambda_{i})_V\subseteq\Gamma$ be representatives
       of the various cosets of $(\lambda_{i})_V$ modulo~$\Delta_{F}$. Thus,
       \[
       (\lambda_{i})_V\, = \,
\textstyle\coprod\limits_{j=1}^{s_{i}}(\gamma_{ij}+\Delta_{F}),
       \]
       where $\coprod$ denotes the disjoint union. For
       $\gamma\in\Gamma_{V}$, we have
       $\varepsilon(\gamma)\in\Lambda_{V}$, hence
       \[
       \varepsilon(\gamma) \, = \, \lambda_{i}+w(a)
       \qquad\text{for some $i\in \{1$, \ldots, $r$\} and some $a\in F^\times$}.
       \]
       It follows that $\gamma-v(a)\in(\lambda_{i})_V$; hence,
       \[
       \gamma-v(a) \ = \ \gamma_{ij}+v(b)
       \qquad\text{for some $j\in \{1$, \ldots, $s_{i}\}$ and some $b\in
       F^\times$}.
       \]
       This shows that $\gamma\equiv\gamma_{ij}\bmod\Gamma_{F}$, hence
       \begin{equation}
	\label{Gamma.eq}
           \Gamma_{V}\ = \ \textstyle\bigcup\limits_{i=1}^r\,
\bigcup\limits_{j=1}^{s_{i}} \, (\gamma_{ij}+\Gamma_{F}).
       \end{equation}
       To complete the proof, it suffices to show the union is disjoint.
       If $\gamma_{ij}\equiv\gamma_{k\ell}\bmod\Gamma_{F}$ for some $i$,
       $j$, $k$, $\ell$, then
       $\varepsilon(\gamma_{ij})\equiv\varepsilon(\gamma_{k\ell})
       \bmod\Lambda_{F}$, hence $i=k$ since
       $\varepsilon(\gamma_{ij})=\lambda_{i}$ and
       $\varepsilon(\gamma_{k\ell})=\lambda_{k}$. Moreover, from
       $\varepsilon(\gamma_{ij})=\varepsilon(\gamma_{k\ell})$ it follows
       that $\gamma_{ij}-\gamma_{k\ell}\in\Delta$, hence
       $\gamma_{ij}\equiv\gamma_{k\ell}\bmod\Gamma_{F}$ implies
       $\gamma_{ij}\equiv\gamma_{k\ell}\bmod\Delta_{F}$, hence also
       $j=\ell$.
\end{proof}

\begin{proposition}
       \label{comp.prop}
       Suppose $\dim_{F}V$ is finite, and let $\lambda_{1}$, \ldots,
       $\lambda_{r}\in\Lambda_{V}$ be representatives of the various
       cosets of $\Lambda_{V}$ modulo $\Lambda_{F}$. The following
       conditions are equivalent:
       \begin{enumerate}
	\item[(i)]
	$\alpha$ is a norm;
	\item[(ii)]
	$\beta$ is a norm and $\alpha_{\lambda}$ is a norm for all
	$\lambda\in\Lambda_V$;
	\item[(iii)]
	$\beta$ is a norm and $\alpha_{\lambda_{i}}$ is a norm for
	$i=1$, \ldots, $r$.
       \end{enumerate}
\end{proposition}

\begin{proof}
       Use the same notation as in the lemma. For simplicity, denote
       $\alpha_{i}=\alpha_{\lambda_{i}}$ and $V_{i}=V^\beta_{\lambda_{i}}$ for
       $i=1$, \ldots, $r$, and
       $V_{ij}=V^\alpha_{\gamma_{ij}}$ for $i=1$, \ldots, $r$ and $j=1$,
       \ldots, $s_{i}$, and use the notation $\DIM{V}{F}$ for
       $\dim_{F}V$. \relax From~\eqref{Gamma.eq} it follows that
       \begin{equation}
	\label{gr0.eq}
	\DIM{\gr_{\alpha}(V)}{\gr_{v}(F)} \ = \ \tsum_{i=1}^r \,
           \tsum_{j=1}^{s_{i}} \,
	\DIM{V_{ij}}{\overline{F}^v}.
       \end{equation}
       Likewise,
       \begin{equation}
	\label{gr1.eq}
	\DIM{\gr_{\beta}(V)}{\gr_{w}(F)} \ = \ \tsum_{i=1}^r \,
           \DIM{V_{i}}{\overline{F}^w}
       \end{equation}
       and
       \begin{equation}
	\label{gr2.eq}
	\DIM{\gr_{\alpha_{i}}(V_{i})}{\gr_{u}(\overline{F}^w) } \ = \
	\tsum_{j=1}^{s_{i}} \, \DIM{V_{ij}}{\overline{F}^v}
	\qquad\text{for $i=1$, \ldots, $r$}.
       \end{equation}
       If $\alpha_{k}$ is not a norm for some $k\in\{1, \ldots,
       r\}$, then
       \[
       \DIM{V_{k}}{\overline{F}^w} \ > \
       \DIM{\gr_{\alpha_{k}}(V_{k})}{\gr_{u}(\overline{F}^w)}.
       \]
       On the other hand, we have
       \[
       \DIM{V_{i}}{\overline{F}^w} \ \geq \
       \DIM{\gr_{\alpha_{i}}(V_{i})}{\gr_{u}(\overline{F}^w)}
       \quad\text{for all $i$};
       \]
       hence, by \eqref{gr1.eq} and \eqref{gr2.eq},
       \[
       \DIM{\gr_{\beta}(V)}{\gr_{w}(F)} \ > \
       \tsum_{i=1}^r \, \DIM{\gr_{\alpha_{i}}(V_{i})}{\gr_{u}(\overline{F}^w)}
        \ = \ \tsum_{i=1}^r \, \tsum_{j=1}^{s_{i}} \,
\DIM{V_{ij}}{\overline{F}^v}.
       \]
       In view of \eqref{gr0.eq}, it follows that
       $\DIM{\gr_{\beta}(V)}{\gr_{w}(F)}>\DIM{\gr_{\alpha}(V)}{\gr_{v}(F)}$.
       Since $\DIM{V}{F}\geq\DIM{\gr_{\beta}(V)}{\gr_{w}(F)}$, we have
       $\DIM{V}{F}>\DIM{\gr_{\alpha}(V)}{\gr_{v}(F)}$, hence $\alpha$ is
       not a norm.

       If each $\alpha_{i}$ is a norm, then
       $\DIM{V_{i}}{\overline{F}^w}=
       \DIM{\gr_{\alpha_{i}}(V_{i})}{\gr_{u}(F)}$ for $i=1$,
       \ldots, $r$, hence \eqref{gr1.eq}, \eqref{gr2.eq}, and
\eqref{gr0.eq} yield
       \[
       \DIM{\gr_{\beta}(V)}{\gr_{w}(F)} \ = \
       \tsum_{i=1}^r \, \tsum_{j=1}^{s_{i}} \, 
\DIM{V_{ij}}{\overline{F}^v} \ = \
       \DIM{\gr_{\alpha}(V)}{\gr_{v}(F)}.
       \]
       It follows that $\alpha$ is a norm if and only if $\beta$ is a
       norm. We have thus proved (i)$\iff$(iii). Since any
       $\lambda\in\Lambda$ can be chosen as a representative of its
       coset, the arguments above also show (i)~$\Rightarrow$~(ii).
       Since (ii)~$\Rightarrow$~(iii) is clear, the proof is complete.
\end{proof}

To set Prop.~\ref{comp.prop} in perspective, we relate the graded
vector spaces $\gr_\alpha(V)$ and $\gr_\beta(V)$ by means of a
value-function-like map
\[
\alpha_*\colon\gr_\beta(V)\to\Gamma\cup\{\infty\}
\]
defined as follows: for $\xi\in\gr_\beta(V)$, $\xi\neq0$, let
$\ell(\xi)$ be the homogeneous component of $\xi$ of least degree, and
let $\lambda=\deg\bigl(\ell(\xi)\bigr)$, so $\ell(\xi)\in
V_\lambda^\beta$; then let
\[
\alpha_*(\xi) \, = \, \alpha_\lambda\bigl(\ell(\xi)\bigr)
\in\lambda \, \subseteq \, \Gamma.
\]
Let also $\alpha_*(0)=\infty$. For $x\in V$ we thus have
\begin{equation}\label{alphastar}
\alpha_*(\tilde x^\beta) \, = \, \alpha(x),
\end{equation}
where $\tilde x^\beta$ denotes the image of $x$ in $\gr_\beta(V)$.

A similar construction applies to the valuation $v$, and yields a map
\[
v_*\colon\gr_w(F) \, \to \, \Gamma\cup\{\infty\},
\]
which satisfies the same properties as a valuation, and such that the
image $v_*(\rho)$ of any nonzero $\rho\in\gr_w(F)$ depends only on its
homogeneous component of least degree. The map $\alpha_*$ deserves the
name of a \emph{graded $v_*$-value function} since it satisfies the
following properties:
\enlargethispage{\baselineskip}
\begin{enumerate}
\item[(i)]
$\alpha_*(\xi)=\infty$ if and only if $\xi=0$; if $\xi\neq0$, then
$\alpha_*(\xi)=\alpha_*\bigl(\ell(\xi)\bigr)$ and
$\varepsilon\circ\alpha_*(\xi)=\deg\ell(\xi)$;
\item[(ii)]
$\alpha_*(\xi+\eta)\geq\min\bigl(\alpha_*(\xi),\alpha_*(\eta)\bigr)$
for $\xi$, $\eta\in\gr_\beta(V)$;
\item[(iii)]
$\alpha_*(\xi\rho)=\alpha_*(\xi)+v_*(\rho)$ for $\xi\in\gr_\beta(V)$
and $\rho\in\gr_w(F)$.
\end{enumerate}
We may thus consider the associated graded structure
$\gr_{\alpha_*}\bigl(\gr_\beta(V)\bigr)$. If $x\in V$ satisfies
$\beta(x)=\lambda$ and $\alpha(x)=\gamma$, we may identify
\[
(x+V^{\beta>\lambda})+\gr_\beta(V)^{\alpha_*>\gamma} \ = \
x+V^{\alpha>\gamma};
\]
thus
\begin{equation}\label{gradeidentity}
\gr_{\alpha_*}\bigl(\gr_\beta(V)\bigr) \, = \ \gr_\alpha(V).
\end{equation}
We define $\alpha_*$ to be a {\it  graded $v_*$-norm} if
$ \ \DIM{\gr_{\alpha_*}\!\big(\grof {\beta}V\big)}
{\gr_{v_*}\!\big({\grof w F}\big)} =
\DIM {\grof {\beta}V}{\grof w F}$.  It is easy to check
that this holds if and only if each $\al_\lambda$ is a $u$-norm. By an
argument analogous to the one in \cite[Prop.~2.5]{RTW} for ungraded norms,
one can check that if $\al_*$ is a graded norm, then for any graded subspace
$\grW$ of $\grof \beta V$, $\alpha_*\rest \grW$ is a graded norm on
$\grW$.  Consequently, by dimension count, the functor
$\gr_{\alpha_*}(\underline{\ \ })$ preserves strict inclusions of
graded subspaces of $\grof \beta V$.
Prop.~\ref{comp.prop} may be rephrased as follows: $\alpha$ is a
$v$-norm if and only if $\beta$ is a $w$-norm and $\alpha_*$ is a
graded $v_*$-norm. Indeed, if $(e_i)_{i=1}^n$ is a splitting base of~
$V$ for $\alpha$, then it is also a splitting base for $\beta$, and
$(\tilde e_i^{\,\beta})_{i=1}^n$ is a splitting base of $\gr_\beta(V)$ for
$\alpha_*$.

We now apply this construction to a finite-dimensional $F$-algebra
$A$. If $\alpha\colon A\to\Gamma\cup\{\infty\}$ is a surmultiplicative
$v$-value function, then the coarser $w$-value function
$\beta=\varepsilon\circ\alpha$ is clearly surmultiplicative, and the
map $\alpha_*$ is also surmultiplicative, by an easy calculation using
\eqref{alphastar}. The notions of gauge and tame gauge for graded norms
are defined analogously to the ungraded cases.

\begin{proposition}
     \label{compgauge.prop}
     The map $\alpha$ is a $v$-gauge (resp.\ a tame $v$-gauge) if and
     only if $\beta$ is a $w$-gauge (resp.\ a tame $w$-gauge) and
     $\alpha_*$ is a graded $v_*$-gauge (resp.\ a tame graded $v_*$-gauge).
\end{proposition}

\begin{proof}
     Prop.~\ref{comp.prop} already shows that $\alpha$ is a $v$-norm if
     and only if $\beta$ is a $w$-norm and $\alpha_*$ is a graded
     $v_*$-norm. We noted above that $\alpha$ is surmultiplicative if and
     only if $\beta$ and $\alpha_*$ are surmultiplicative.

     Suppose $\alpha$ is a $v$-gauge. Since
     $\gr_{\alpha_*}\bigl(\gr_\beta(A)\bigr)=\gr_\alpha(A)$ and
     $\gr_\alpha(A)$ is semisimple, it follows that $\gr_\beta(A)$ is
     semisimple. For, if $\grI$
     is a nontrivial nilpotent homogeneous left  ideal of $\gr_\beta(A)$,
then $\gr_{\alpha_*}(\grI)$ is a nontrivial nilpotent homogeneous left
     ideal of $\gr_{\alpha_*}\bigl(\gr_\beta(A)\bigr)$.
    Thus, $\beta$ is
     a $w$-gauge. Also, $\gr_{\alpha_*}\bigl(\gr_\beta(A)\bigr)$ is
     semisimple by hypothesis, hence $\alpha_*$ is a graded
     $v_*$-gauge. Conversely, if $\beta$ is a $w$-gauge and $\alpha_*$ is
     a graded $v_*$-gauge, then $\alpha$ is a $v$-gauge since
     $\gr_\alpha(A)=\gr_{\alpha_*}\bigl(\gr_\beta(A)\bigr)$.

Assume now that $\al$ is a $v$-gauge. For the centers we have the obvious
inclusions
\begin{equation}\label{centers}
     \gr_\alpha\bigl(Z(A)\bigr) \ = \ \gr_{\alpha_*}\bigl(\gr_\beta(Z(A))\bigr)
      \ \subseteq \  \gr_{\alpha_*}\bigl(Z(\gr_\beta(A))\bigr)  \ \subseteq \
     Z\bigl(\gr_{\alpha_*}(\gr_\beta(A))\bigr)
      \ = \ Z\bigl(\gr_\alpha(A)\bigr).
\end{equation}
Thus, $Z\bigl(\gr_\alpha(A)\bigr) = \gr_\alpha\bigl(Z(A)\bigr)$ if and only if
we have equalities throughout \eqref{centers}; since
$\gr_{\alpha_*}(\underline{\ \ })$ preserves strict inclusions, this holds if
and only if
\[
\gr_\beta\bigl(Z(A)\bigr)=Z\bigl(\gr_\beta(A)\bigr)\quad\text{and}\quad
\gr_{\alpha_*}\bigl(Z(\gr_\beta(A))\bigr)=
Z\bigl(\gr_{\alpha_*}(\gr_\beta(A))\bigr).
\]
Assume we have these equalities.  Let $Z = Z(A)$, which is
a direct product of fields, as $A$ is semisimple.
The separability condition on the graded center required for tameness
holds for $\al$ if and only if it holds for $\al_*$, since they
have the same graded rings.  Suppose now that
$\grof\beta Z$ is not separable over $\grof w F$.  Because
$\grof \beta A$ is semisimple, its center $\grof \beta Z$ is a
direct product $\grC_1\times \ldots \times \grC_k$ of graded fields, and
some $\grC_j$ must not be separable over $\grof w F$. By
\cite[Prop.~3.7, Prop.~3.5]{HW1} there is a graded field $\grT$ with
$\grof w F \subseteq \grT \subsetneqq \grC_j$ and $\grC_j$
purely inseparable over $\grT$.  So,
\[
\grof v F  \, = \,
\gr_{\alpha_*}\!\big(\grof wV\big)   \subseteq \,  \grof{\alpha_*}\grT
    \, \subsetneqq  \, \grof{\alpha_*}{\grC_j},
\]
and $\grof{\alpha_*}{\grC_j}$
is purely inseparable over $\grof{\alpha_*}\grT$.
Now, $\grof \al Z = \gr_{\al_*}\!\big(\grof \beta Z\big) =
\prod_{i=1}^k\grof{\al_*}{\grC_i}$.  Since $\grof{\al_*}{\grC_j}$
is purely inseparable over $\grof{\al_*}\grT$, it cannot be separable over
$\grof v F$, so $\grof\al Z$ is not separable over $\grof v F$.  Thus,
$\grof \al Z$ is separable over $\grof v F$ if and only
$\gr_{\al_*}\!\big(\grof \beta Z\big)$  is separable over $\gr_{v_*}
\!\big(\grof w F)$ and $\grof \beta Z$ is separable over $\grof w F$.
    Therefore, $\alpha$ is a tame $v$-gauge
if and only if $\beta$ is a tame $w$-gauge and $\alpha_*$ is a tame
graded $v_*$-gauge.
\end{proof}

\section{Descent of norms}
\label{descent.sec}

Throughout this section, we fix the following notation: $V$ is a
finite-dimensional vector space over a field $F$, and
$v\colon F\to\Gamma\cup\{\infty\}$ is a valuation. Let $(F_h,v_h)$
be a Henselization of $(F,v)$. If $\alpha\colon
V\otimes_FF_h\to\Gamma\cup\{\infty\}$ is a $v_h$-norm, then clearly
$\alpha\rvert_V\colon V\to\Gamma\cup\{\infty\}$ is a $v$-value
function, but not necessarily a $v$-norm unless $\Gamma$ has rank one,
see Prop.~\ref{rk1.prop} and Ex.~\ref{notnorm.ex}. In this section, we
give an inductive criterion for $\alpha\rvert_V$ to be a
$v$-norm when $\Gamma$ is the divisible hull of $\Gamma_F$ and the
rank $\rk(\Gamma)$ is finite, see Prop.~\ref{normcrit.prop}.

We first discuss the descent problem in a general context: let
$(K,v_K)$ be an arbitrary valued field extension of $(F,v)$, and let
$\alpha\colon V\otimes_FK\to\Gamma\cup\{\infty\}$ be a $v_K$-norm. We
identify $V$ with its canonical image in $V\otimes_FK$. For any $x\in
V$ and $c\in K$ we have
\[
\al(x\otimes c) = \al((x\otimes 1) \cdot c) = \al \rest V(x) + v_K(c).
\]
Therefore, for any $\gamma \in\Gamma_{V,\al\rest V}$ and $\delta
\in \Gamma_K$ the  usual $F$-bilinear map $V\times K \to V\otimes _F K$ sends
$V^{\ge \gamma} \times K^{\ge \delta}$ into $(V\otimes_F K)^{\ge \gamma
+ \delta}$.  Likewise, $V^{> \gamma} \times K^{\ge \delta}$ and
$V^{\ge \gamma} \times K^{> \delta}$ map into
$(V\otimes_F K)^{> \gamma+ \delta}$.
Consequently, there is a well-defined induced map
$V_\gamma \times K_\delta  \to (V\otimes_F K)_{\gamma+\delta}$ given by
$(\tilde x, \tilde c) \mapsto \tilde{x\otimes c}$.   The direct sum of these
maps over all such $\gamma, \delta$ yields a map
${\grof {\al\rest V} V\times\grof {v_K} K \to \grof\al{V\otimes_F K}}$,
which is clearly $\grof v F$-bilinear, hence there is a canonical map
\[
\chi\colon\gr_{\alpha\rvert_V}(V)\otimes_{\gr_v(F)} \gr_{v_K}(K) \to
\gr_\alpha(V\otimes_FK)
\]
which maps $\tilde x\otimes\tilde c$ to $\tilde{x\otimes c}$ for $x\in
V$ and $c\in K$.

On the other hand, recall from Sec.~\ref{sec:compa} (see
\eqref{eq:griso}) that if $\alpha\rvert_V$ is a $v$-norm on $V$, then
there is a canonical isomorphism of $\gr_{v_K}(K)$-vector spaces
\[
\rho\colon \gr_{\alpha\rvert_V\otimes v_K}(V\otimes_FK) \iso
\gr_{\alpha\rvert_V}(V)\otimes_{\gr_v(F)}\gr_{v_K}(K)
\]
which maps $\tilde{x\otimes c}$ to $\widetilde{x}\otimes\widetilde{c}$
for $x\in V$ and $c\in K$.

\begin{lemma}\label{reconstruct}
The following conditions are equivalent:
\begin{enumerate}
\item[(a)]
$\al \rest V$ is a $v$-norm on $V$ and $\al = \al \rest V \otimes v_K$.
    \item[(b)]
$V$ contains a $K$-splitting base of the norm $\al$  on $V\otimes_F K$.
    \item[(c)]  $\al \rest V$ is a $v$-norm and the canonical map
$\chi$ is injective.
\end{enumerate}
When these conditions hold, the map $\chi$ is a graded isomorphism,
which is the inverse of $\rho$, and
$\Gamma_{V\otimes_F K, \al} = \Gamma_{V, \al\rest V} + \Gamma_{K,v_K}$.
\end{lemma}

\begin{proof}
(b) $\Rightarrow$  (a) If $\mathcal B = (e_i)_{i=1}^n \subseteq V$
is a splitting base for $\al$ on $V\otimes _F K$, then $\mathcal B$ is clearly
also a splitting base for $\al \rest V$ on $V$.  So, $\al\rest V$ is
a $v$-norm.
Furthermore, by the definition
of $\al\rest V \otimes v_K$, we have for any $k_1, \ldots, k_n \in K$,
\[
\begin{aligned}
(\al\rest V \otimes v_K)\big(\tsum_{i=1}^n e_i \otimes k_i\big) \ &= \
\min\limits_{1\le i\le n}\big(\al\rest V(e_i) + v_K(k_i) \big) \ = \
\min\limits_{1\le i\le n}\big( \al(e_i\otimes 1) + v_K(k_i)\big) \\  &= \
\al \big(\tsum_{i=1}^n (e_i\otimes 1)\cdot k_i\big)  \ = \
\al(\tsum_{i=1}^n e_i\otimes k_i\big),
\end{aligned}
\]
showing that $\al\rest V\otimes v_K = \al$.

(a)  $\Rightarrow$ (c) \
When (a) holds, $\al \rest V$ is a norm, and $\chi$
is clearly the inverse of $\rho$, so $\chi$ is injective.

(c)  $\Rightarrow$ (b) \   Suppose (c) holds.  Let $(e_i)_{i=1}^n$ be an
$F$-splitting   base for $\al \rest V$ on $V$.  Then, by
\cite[Cor.~2.3(ii)]{RTW}
$\tilde{e_1}, \ldots, \tilde {e_n}$ are $\grof v F$-linearly independent
in $\grof {\al\rest V} V$.   Hence, $\tilde {e_1} \otimes \tilde 1,
\ldots, \tilde {e_n} \otimes \tilde 1$ are $\grof {v_K} K$-linearly independent
in $\grof {\al \rest V}V \otimes
_{\grof v F} \grof {v_K} K$.  By the injectivity of $\chi$ the
$\chi(\tilde{e_i}\otimes \tilde 1) = \tilde{e_i \otimes 1}$ are
$\grof {v_K}K$-linearly independent in $\grof \al {V\otimes_F  K}$. But, since
$\al $ and $\al\rest V$ are norms,
\[
\DIM{\grof \al{V\otimes_F K}}{\grof {v_K}K}\,=\,
\DIM{V\otimes _FK}{K}\,=\,\DIM{V}{F}\,=\,n.
\]
So,  $\big(\tilde{e_i \otimes 1}\big)_{i = 1}^n$ is a homogeneous
$\grof {v_K} K$-vector space base of $\grof \al
{V\otimes_F K}$, hence $(e_i\otimes 1)_{i = 1}^n$ is a $K$-splitting base for
$\al$ on $V \otimes_F K$ by \cite[Cor.~2.3(ii)]{RTW}.

When the conditions (a) -- (c) hold, we have
\[
\Gamma_{V\otimes_F K, \al} \ = \  
\Gamma_{V\otimes_F K, \al \rest V \otimes v_K} 
 \ = \  \Gamma_{V, \al\rest V}
+ \Gamma_{K,v_K}
\]
and the map $\chi$ is the inverse of $\rho$, so $\chi$ is an isomorphism.
\end{proof}

Note that under the hypotheses of Lemma~\ref{reconstruct} if
$\al\rest V$ is a norm then $\al \ge \al\rest V \otimes v_K$.  For, if
$(e_i)_{i=1}^n$ is a $v$-splitting base for $\al\rest V$ on $V$, then for
any $k_1$, \ldots, $k_n\in K$,
\[
\al \big(\tsum _{i=1}^n e_i \otimes k_i \big) \ \ge \
\min\limits_{1\le i\le n}\big(\al(e_i \otimes k_i)\big)  \ = \
\min\limits_{1\le i\le n}\big(\al(e_i) +v_K(k_i)\big) \ = \
(\al\rest V \otimes v_K) \big(\tsum _{i=1}^n e_i \otimes k_i \big) .
\]
We next show that the inequality $\al \ge \al\rest V \otimes v_K$ is
actually an equality when $K$ is immediate over $F$, but not in general.

\begin{corollary}
     \label{immediate.lem}
     Let $(K,v_K)$ be an immediate valued field extension of $(F,v)$ and let
     ${\alpha\colon V\otimes_FK\to\Gamma\cup\{\infty\}}$ be a
     $v_K$-norm. If $\alpha\rvert_V$ is a norm, then
     $\alpha=\alpha\rvert_V\otimes v_K$. So, the canonical map
     $\gr_{\alpha\rvert_V}(V)\to\gr_\alpha(V\otimes_FK)$ is an isomorphism
     $\gr_{\alpha\rvert_V}(V)\cong\gr_\alpha(V\otimes_FK)$,
     and $\Gamma_V=\Gamma_{V\otimes_FK}$.
\end{corollary}

\begin{proof}
Since $v_K$ is immediate over $v$, we have $\grof {v_K}K = \grof vF$, so the
canonical map $\chi$ of Lemma~\ref{reconstruct}(c) is just the injection
$\grof{\al\rest V}V \hookrightarrow \grof {\al}{V\otimes_F K}$ arising
from the canonical inclusion $V\hookrightarrow V\otimes_F K$.  Thus,
the corollary follows from Lemma~\ref{reconstruct}, using $\Gamma_K =
\Gamma_F$ for the last assertion.
\end{proof}

\begin{example}
     Let $(K,v_K)$ be an extension of $(F,v)$ with
     $\overline{F}\subsetneqq\overline{K}$. Let $\xi\in K$ be such that
     $v_K(\xi)=0$ and $\overline{\xi}\notin\overline{F}$, and let $V$ be a
     $2$-dimensional $F$-vector space with base $(e_1,e_2)$.
Let $f = e_1\otimes1+e_2\otimes\xi \in V\otimes_F K$, and consider the
     $v_K$-norm $\alpha$ on $V\otimes_FK$ with splitting base
     $(e_1\otimes1,f)$ such that
     \[
     \alpha(e_1\otimes1)=0\qquad\text{and}\qquad
     \alpha(f)>0.
     \]
     Then, as $e_2 = (f-e_1)\xi^{-1}$, we have for $c_1$, $c_2\in F$
     \begin{align*}
     \alpha\rvert_V(e_1c_1+e_2c_2) \ &= \
    \al\big((e_1(c_1 - \xi^{-1}c_2) + fc_2\big) \ = \
\min\big(\al(e_1)+v_K(c_1-\xi^{-1}c_2), \al(f) + v_K(c_2)\big)\\
    \ &= \ \min\bigl(v(c_1),v(c_2)\bigr).
    \end{align*}
     Hence, $(e_1,e_2)$ is a $v$-splitting base of $V$ for
     $\alpha\rvert_V$, showing that $\al \rest V$ is a
$v$-norm on $V$.
    However, $\alpha\rvert_V\otimes v_K < \al$
since $(\al\rest V\otimes v_K)(f) = \min\big(v_K(1), v_K(\xi)\big) = 0
<\al(f)$. Thus, the first condition in Lemma~\ref{reconstruct}(a)
holds, but not the second.  
The Lemma shows that $V$ does not contain any splitting base for 
the $v_K$-norm $\alpha$ on $V\otimes_FK$.  
Also, the second condition in part (c) of the Lemma fails, since  
the canonical map $\chi$ satisfies
\[
\chi(\tilde{e_1}\otimes1+\tilde{e_2}\otimes\tilde{\xi}) \ = \ 
\tilde{e_1\otimes1} +\tilde{e_2\otimes\xi} \, = \, 0.
\]
\end{example}

We now turn to the descent problem posed at the beginning of this
section, for $(K,v_K)=(F_h,v_h)$ a Henselization of $(F,v)$. The rank
one case is easy:

\begin{proposition}
       \label{rk1.prop}
       Let
       $\alpha\colon V\otimes_{F}F_{h}\to\Gamma\cup\{\infty\}$ be a 
$v_{h}$-norm,
let $\gamma\in \Gamma$, and suppose
        ${\im(\alpha) \subseteq
\gamma+(\Gamma_{F}\otimes_{\mathbb{Z}}\mathbb{Q})}$.
       If $\rk(\Gamma_F)=1$, then $\alpha\vert_{V}$ is a $v$-norm and
$\al = \al \rest V \otimes v_h$.
\end{proposition}

\begin{proof}
       Let $(e_{i})_{i=1}^n$ be an arbitrary $F$-base of $V$ and let
       $x\in V\otimes_{F}F_{h}$,
       \[
       x \, = \, \tsum_{i=1}^ne_{i}\otimes k_{i}
       \qquad\text{for some $k_{i}\in F_{h}$.}
       \]
       Since $\rk(\Gamma_{F})=1$, the field $F$ is dense in $F_{h}$ for
       the topology of the valuation $v_{h}$: see \cite[\S1.6]{Er} or
       use the fact that $F$ is dense in its completion $\widehat F$ and
       that $F_{h}$ embeds in $\widehat F$ by \cite[Th.~17.18]{E}.
Furthermore, $\Gamma_F$ is dense in its divisible hull
$\Gamma_F \otimes_{\mathbb{Z}}\mathbb{Q}$.
For each $i$, $1\le i\le n$, since $\al(x) - \al(e_i\otimes 1)
\in \Gamma_F \otimes_{\mathbb{Z}}\mathbb{Q}$, we may therefore find
       an element $f_{i}\in F$ such that
       \[
       v_{h}(k_{i}-f_{i}) \ > \ \alpha(x)-\alpha(e_{i}\otimes1).
       \]
       Let
       $y=\sum_{i=1}^ne_{i}\otimes f_{i}=
       \sum_{i=1}^ne_{i}f_{i}\otimes1\in V$. Then,
       \[
       \alpha(x-y) \ = \ \alpha\bigl(
       \tsum_{i=1}^ne_{i}\otimes(k_{i}-f_{i})\bigr)
        \ \geq \
       \min \limits_{1\leq i\leq
       n}\bigl(\alpha(e_{i}\otimes(k_{i}-f_{i}))\bigr)
        \ = \ \min\limits_{1\le i\le n}\big(\al(e_i) + v_h(k_i - f_i) \big)
>    \ \alpha(x).
       \]
       Hence,
       \[
       \tilde x \, = \, \tilde y \, \in \, \gr_{\alpha\vert_{V}}(V).
       \]
       This proves that the monomorphism
       $\gr_{\alpha\vert_{V}}(V)\hookrightarrow\gr_{\alpha}(V\otimes_{F}F_{h})$
is an isomorphism. Hence,
as $\al$ is a norm,
\[
\DIM{\grof {\al\rest V}V \,}{\grof v F}  \ = \
\DIM{\grof\al{V\otimes_F F_h} \, }
{\grof{v_h}{F_h}}  \ = \  \DIM{(V\otimes _F F_h)}{F_h}  \, = \,  \DIM V F,
\]
which shows that     $\alpha\vert_{V}$ is a $v$-norm.
\end{proof}

Now, suppose $\Gamma=\Gamma_F\otimes_\mathbb{Z}\mathbb{Q}$, with
$\rk(\Gamma)>1$, and suppose $\Gamma$ contains a convex subgroup $\Delta$ of
rank~$1$. As in \S\ref{comp.sec}, we consider the canonical map
$\varepsilon\colon\Gamma\to\Gamma/\Delta=\Lambda$ and the coarser
valuation
\[
w \, = \, \varepsilon\circ v\colon F\to\Lambda\cup\{\infty\}.
\]
Let $(F_{h,v},v_h)$ be a Henselization of $(F,v)$ and $(F_{h,w},w_h)$
a Henselization of $(F,w)$. Let also
\[
y \, = \, \varepsilon\circ v_h\colon F_{h,v}\to\Lambda\cup\{\infty\}.
\]
By \cite[Cor.~4.1.4, p.~90]{EP}, the valuation $y$ is Henselian, hence we
may assume $(F_{h,w},w_h)\subseteq(F_{h,v},y)$.

Let $\alpha\colon V\otimes_FF_{h,v}\to\Gamma\cup\{\infty\}$ be a
$v_h$-norm, and let
\[
\beta \, = \, \varepsilon\circ\alpha\colon
V\otimes_FF_{h,v}\to\Lambda\cup\{\infty\}.
\]
By Prop.~\ref{comp.prop}, the map $\beta$ is a $y$-norm.

\begin{proposition}
     \label{normcrit.prop}
     If $\beta\rvert_V$ is a $w$-norm and $\beta=\beta\rvert_V\otimes y$,
     then $\alpha\rvert_V$ is a $v$-norm on $V$ and $\al = \al \rest V
\otimes v_h$.
\end{proposition}

\begin{proof}
     As observed in \S\ref{comp.sec}, the valuation $v$ induces a
     valuation $u$ on the residue field $\overline{F}^w$,
     \[
     u\colon\overline{F}^w\to\Delta\cup\{\infty\}.
     \]
Note that the value group of $u$ is $\Delta_F = \Gamma_F \cap \Delta$ and, as
$\Delta$ is divisible and torsion-free,
\[
\Delta_F\otimes_\mathbb{Z}\mathbb{Q} \ = \
\big(\Gamma_F\otimes_\mathbb{Z}\mathbb{Q}
\big)   \cap   \big(\Delta\otimes_\mathbb{Z}\mathbb{Q}\big) \ = \
\Gamma \cap \Delta \, = \, \Delta.
\]
     Let
     $\Lambda_V=\beta\rvert_V(V\setminus\{0\})\subseteq\Lambda$. Clearly,
     $\beta\rvert_V=\varepsilon\circ(\alpha\rvert_V)$. In order to show
     $\alpha\rvert_V$ is a norm, it therefore suffices, by
     Prop.~\ref{comp.prop}, to show that each map
\begin{equation} \label{alpharest}
     (\alpha\rvert_V)_\lambda\colon V_\lambda^{\beta\rvert_V}\to
     \lambda\cup\{\infty\}, \qquad\text{for $\lambda\in\Lambda_V$},
\end{equation}
     is a $u$-norm. To simplify notation, we write $\Vlam$ for
$V_\lambda^{\beta\rest V}$.
Note that the canonical inclusion $V \hookrightarrow
V\otimes_F F_{h,v}$ is compatible with the respective value functions
$\beta \rest V$ and $\beta$ so yields an injection $\Vlam
\hookrightarrow(V\otimes _F F_{h,v})_\lambda^\beta$; let
$V_\lambda'$ denote the image of $V_\lambda$.
     Then, clearly
$\al_\lambda \rest{\Vlam'} \cong (\al\rest V)_\lambda$.

     Let $u_h\colon\overline{F_{h,v}}^y\to\Delta\cup\{\infty\}$ be the
     valuation induced by $v_h$. As observed by Morandi \cite[p.~239]{M},
     $(\overline{F_{h,v}}^y,u_h)$ is a Henselization of
     $(\overline{F}^w,u)$. Since $\alpha$ is a $v_h$-norm,
     Prop.~\ref{comp.prop} shows that
     \[
     \alpha_\lambda\colon (V\otimes_FF_{h,v})_\lambda^\beta
     \to\lambda\cup\{\infty\}
     \]
     is a $u_h$-norm for every $\lambda\in\Lambda_V$. Since
      $(F_{h,v},y)$ is an inertial
     extension of $(F_{h,w},w_h)$ by \cite[p.~239]{M}, we have
\[
\grof y{F_{h,v}} \ \cong \ \grof {w_h} {F_{h,w}}
\otimes_{\grof {w_h} {F_{h,w}}_0}\grof y {F_{h,v}}_0
\ = \ \grof {w} {F}
\otimes_{\grof {w} {F}_0}\grof y {F_{h,v}}_0.
\]
    Because $\beta=\beta\rvert_V\otimes y$, this yields graded isomorphisms
\begin{align*}
\grof\beta{V\otimes_FF_{h,v}} \
&\cong \ \grof{\beta\rest V} V \otimes_{\grof w F} \grof y {F_{h,v}} \
\cong \ \grof{\beta\rest V} V \otimes_{\grof w F_0}\grof y {F_{h,v}}_0 \\
&\cong \ \grof{\beta\rest V} V\otimes_{\overline F^w}\overline {F_{h,v}}^y.
\end{align*}
For any $\lambda \in \Lambda_F$, when we restrict these graded isomorphisms
to the
$\lambda$-component we obtain the $\overline {F_{h,v}}^y$-vector space
isomorphism
     \[\psi\colon
     (V\otimes_FF_{h,v})_\lambda^\beta \ \iso \
     \Vlam\otimes_{\overline{F}^w}\overline{F_{h,v}}^y.
     \]
Let $\wh\al = \al_\lambda \circ \psi^{-1}\colon
\Vlam\otimes_{\overline{F}^w}\overline{F_{h,v}}^y
\to \lambda \cup \{\infty\}$, which is  the $u_h$-value function on $\im(\psi)$
corresponding to~$\al_\lambda$~on the domain of $\psi$. Since $\al_\lambda$
is a $u_h$-norm, so is $\wh \al$.  Because $(\ov{F_{h,v}}^y, u_h)$ is
a Henselization of~$(\ov F^w,u)$ and $\lambda$ is a coset of
$\Delta = \Delta_F \otimes_\mathbb{Z}\mathbb{Q}$, which has rank $1$, with
$\Delta_F$ the value group of $u$,  Prop.~\ref{rk1.prop} applies to~$\wh \al$,
and shows that $\wh\al\rest{\Vlam}$ is a $u$-norm.
Note that $\psi$ maps the $V_\lambda'$ defined above after
\eqref{alpharest} to the copy of $V_\lambda$ in $\im(\psi)$.
So, $\al_\lambda\rest{V_\lambda'} \cong \wh\al\rest \Vlam$. But, we saw
above that $(\al\rest V)_\lambda \cong \al_\lambda\rest{V_\lambda'}$.
Since $\wh\al\rest{\Vlam}$ is a $u$-norm, these isomorphisms show that
$(\al\rest V)_\lambda$ is also a $u$-norm.
Thus, by Prop.~\ref{comp.prop} $\al\rest V$ is a $v$-norm; then
$\al = \al \rest V \otimes v_h$ by Cor.~\ref{immediate.lem}.
\end{proof}

The following is an example of a norm on a Henselization that does not descend
to a norm.

\begin{example}
     \label{notnorm.ex}
     Let $k$ be any field with $\charac (k)\neq2$, and let $F=k(x,y)$ with
     $x$ and $y$ algebraically independent over $k$. Let $v$ be the
     valuation on $F$ obtained by restriction from the canonical
     Henselian valuation on $k((x))((y))$, so
     $\Gamma_F=\mathbb{Z}\times\mathbb{Z}$ and $\overline{F}=k$. Let
     $(F_h,v_h)$ be a Henselization of $(F,v)$. Let
     $A=\Bigl(\frac{1+x,\, y}{F}\Bigr)$, a quaternion division algebra over
     $F$, and let $A_h=A\otimes_FF_h$. The algebra $A_h$ is split since
     $1+x\in F_h^{\times2}$. Therefore, we may find $v_h$-gauges on $A_h$
     that are unramified, in the sense that $\Gamma_{A_h}=\Gamma_F$. Fix
     such a $v_h$-gauge $\alpha$. We claim that $\alpha\rvert_A$ is not a
     $v$-norm on $A$.

     Suppose the contrary. Then $\gr_{\alpha\rvert_A}(A)=\gr_\alpha(A_h)$
     by Lemma~\ref{immediate.lem}, so $\alpha\rvert_A$ is a
     $v$-gauge. Consider the convex subgroup
     $\Delta=\mathbb{Z}\times\{0\}\subseteq\Gamma_F$ and the canonical
     epimorphism
     \[
     \varepsilon\colon\Gamma_F\to\Gamma_F/\Delta \, = \, \mathbb{Z}.
     \]
     Let $w=\varepsilon\circ v\colon F\to\mathbb{Z}\cup\{\infty\}$, which
     is the $y$-adic valuation on $F$, and
     $\beta=\varepsilon\circ\alpha\colon
     A_h\to\mathbb{Z}\cup\{\infty\}$. Proposition~\ref{compgauge.prop}
     shows that $\beta\rvert_A$ is a tame $w$-gauge on $A$. However, the
     $y$-adic valuation $w$ extends to $A$, so by \cite[Cor.~3.4]{TWgr}
     $\beta\rvert_A$ is the (unique) valuation on $A$ that extends
     $w$. In particular, if $j\in A$ satisfies $j^2=y$ we must have
     $\beta\rvert_A(j)=\frac12$. This is a contradiction since
     $\beta\rvert_A$ is unramified.
\end{example}

\section{Non-Henselian valuations}
\label{nonHensel.sec}

Let $(F,v)$ be a valued field and let $A$ be a finite-dimensional simple
$F$-algebra with an involution $\sigma$. Let $K=Z(A)$, and assume
$F$ is the subfield of $K$ fixed under $\sigma$. Fix a
Henselization $(F_h,v_h)$ of $(F,v)$.

\begin{theorem}
     \label{mainnonHensel.thm}
Suppose $A$ is split by the
maximal tamely ramified extension of $F_h$. Moreover, if
${\charac(\overline{F})=2}$ suppose that $\sigma$ is not an orthogonal
involution.Then, the
  following conditions are equivalent:
     \begin{enumerate}
     \item[(a)]
     $\sigma\otimes\id_{F_h}$ is an anisotropic involution on $A\otimes _F F_h$;
     \item[(b)]
     there exists a $\sigma$-special $v$-gauge $\varphi$ on $A$
i.e.,
     $\varphi(\sigma(x)x)=2\varphi(x)$ for all $x\in A$.
     \end{enumerate}
     When they hold,  $\varphi$ is the unique $v$-gauge on
     $A$ that is invariant under $\sigma$, it is tame, and
     its value group lies in the divisible hull of $\Gamma_F$.
\end{theorem}

\begin{proof}
     Let $A_h=A\otimes_FF_h$ and
     $\sigma_{h}=\sigma\otimes\id_{F_{h}}$. If $\varphi$ is a 
$\sigma$-special $v$-gauge
     on $A$, then by Prop.~\ref{eqcond.prop}
$\varphi$~is~invariant under $\sigma$ and $\tilde\sigma$ is anisotropic
     on $\gr_\varphi(A)$. By Cor.~\ref{cor:compatext}, $\varphi\otimes v_h$
is invariant under $\sigma_h$.
Since
     $\gr_{\varphi\otimes v_h}(A_h)\cong\grof\varphi A \otimes _{\grof vF}
\grof{v_h}{F_h} \cong\gr_\varphi(A)$ and
     $\tilde{\sigma_h}\cong\tilde\sigma$, it follows that
$\tilde{\sigma_h}$ is anisotropic, hence $\sigma_h$~must also be
     anisotropic, proving (b)~$\Rightarrow$~(a).

     Now, suppose (a) holds. Let $\varphi_1$ and $\varphi_2$ be
     $v$-gauges on $A$ that are each invariant under $\sigma$.   Then,
by Cor.~\ref{cor:compatext} each  $\varphi_i\otimes v_h$
is a surmultiplicative $v_h$-norm
     on $A_h$ which is invariant under $\sigma_h$.
Moreover, $\varphi_i\otimes v_h$ is a gauge on $A_h$
since $\gr_{\varphi_i\otimes v_h}(A_h)
\cong\grof{\varphi_i} A \otimes _{\grof vF}
\grof{v_h}{F_h} \cong\gr_{\varphi_i}(A)$
and $\varphi_i$ is a gauge on~$A$.
Since $\sigma_h$~is assumed anisotropic, the uniqueness part of
     Th.~\ref{mainHensel.thm}  (applied to $\sigma_h$ on $A_h$)
    yields $\varphi_1\otimes
     v_h=\varphi_2\otimes v_h$, hence
     $\varphi_1=\varphi_2$. Th.~\ref{mainHensel.thm}  also shows that
     $\varphi_1\otimes v_h$ is tame and satisfies
     $\varphi_1\bigl(\sigma_h(x)x\bigr)=2\varphi_1(x)$ for all $x\in
     A_h$, hence $\varphi_1$ is tame and satisfies  condition~(b).
Furthermore, $\Gamma_{A, \varphi_1} = \Gamma_{A_h, \varphi_1\otimes v_h}$
which lies in the divisible hull of $\Gamma_{F_h} = \Gamma_F$ by
Th.~\ref{mainHensel.thm}.

     Thus, it only remains to prove the existence of a $v$-gauge on $A$
     invariant under $\sigma$, assuming $\sigma_h$ is anisotropic.
Note first that $K\otimes_F F_h$ is a field.
For, otherwise, as $K$ is Galois over $F$ with $\DIM KF = 2$,
$K\otimes_F F_h$ would be a direct sum of two fields, and the nontrivial
$F_h$-automorphism $\sigma\rest{K\otimes_F F_h}$ must permute the two
primitive idempotents of $K\otimes_F F_h$, call them $e_1$ and $e_2$.
Then, $\sigma_h(e_1)e_1 = e_2e_1 =0$;  but, this cannot happen as $\sigma_h$ is
anisotropic.  Since $K\otimes _F F_h$ is a field and $K = Z(A)$,
$A_h \cong A\otimes _K(K\otimes _F F_h)$ is a central simple
$K\otimes _F F_h$-algebra.

    Because $A_h$ is simple, $\sigma_h$ is anisotropic, and
$v_h$ is Henselian, Th.~\ref{mainHensel.thm}  yields a
     $\sigma_h$-invariant $v_h$-gauge $\varphi_h$ on $A_h$ whose value
     set lies in the divisible hull of $\Gamma_{F_h} = \Gamma_F$.
The restriction
     $\varphi=\varphi_h\rvert_A$ is clearly a $\sigma$-invariant
     $v$-value function whose value set lies in the divisible hull of
     $\Gamma_F$. Henceforth, we may thus assume
     $\Gamma=\Gamma_F\otimes_\mathbb{Z}\mathbb{Q}$. If we show
that $\varphi$
     is a $v$-norm, then Cor.~\ref{immediate.lem} yields
$\varphi_h = \varphi \otimes v_h$, so
     $\gr_\varphi(A)=\grof{\varphi\otimes v_h} {A_h} =
\gr_{\varphi_h}(A_h)$, hence $\varphi$ is a
     $v$-gauge, and the proof will be  complete.

     Suppose first that $\rk(\Gamma_F)<\infty$. We then argue by induction on
     $\rk(\Gamma_F)$. If $\rk(\Gamma_F)=1$, then
     Prop.~\ref{rk1.prop} shows that $\varphi$ is a $v$-norm. So, we may
assume
     $\rk(\Gamma_F)>1$.  Let $\Delta\subseteq\Gamma$ be a convex subgroup
     of rank~$1$ and let
     \[
     \varepsilon\colon\Gamma\to\Lambda \, = \, \Gamma/\Delta
     \]
     be the canonical epimorphism. Let $w=\varepsilon\circ v$ and
     $y=\varepsilon\circ v_h$, to agree with the notation of
     \S\ref{descent.sec}. So, $w$~has value group $\Lambda_F =
(\Gamma_F +\Delta)/\Delta$, and $\Lambda = \Lambda_F\otimes_
\mathbb{Z}\mathbb{Q}$, which has rank $\rk(\Gamma_F) -1$.
Let  $(F_{h,w},w_h)\subseteq(F_h,y)$ be a
     Henselization of $(F,w)$. Since $\sigma_h$ is anisotropic, its restriction
      $\sigma\otimes\id_{F_{h,w}}$ is an anisotropic  involution on the subring
     $A\otimes_FF_{h,w}$ of $A_h$. Since $A_h\cong (A\otimes_FF_{h,w})
\otimes_{F_{h,w}} F_h$ and $A_h$ is simple, $A\otimes_FF_{h,w}$ must
also be simple.
Therefore, Th.~\ref{mainHensel.thm} applies,  yielding  a
     $w_h$-gauge $\psi_h$ on $A\otimes_FF_{h,w}$  invariant under
     $\sigma\otimes\id_{F_{h,w}}$.  By induction,
     $\psi_h\rvert_A$ is a $w$-gauge on $A$ invariant under
     $\sigma$. The same argument as for $\varphi_1$ above shows that
the gauge $\psi\rest A$ is tame.  Therefore, by \cite[Cor.~1.26]{TWgr}
    $\psi_h\rvert_A\otimes y$ is a
     $y$-gauge on $A_h$, which is $\sigma_h$-invariant by
Cor.~\ref{cor:compatext}.  But, $\varepsilon\circ\varphi_h$ is
also a $y$-gauge on $A_h$, by Prop.~\ref{compgauge.prop} since
$\varphi$ is a gauge, and $\varepsilon\circ\varphi_h$ is invariant
under $\sigma_h$ because $\varphi_h$ is.
    By the uniqueness given in
     Th.~\ref{mainHensel.thm}, it follows that
     $\varepsilon\circ\varphi_h=\psi_h\rvert_A\otimes y$. Restricting to
     $A$, we also have
     $\varepsilon\circ\varphi=\psi_h\rvert_A$, which is a $w$-gauge so
a $w$-norm on $A$. Furthermore,
$(\varepsilon\circ\varphi) \otimes y = \psi_h\rvert_A \otimes y =
\varepsilon\circ\varphi_h$.
    Prop.~\ref{normcrit.prop} with $\al = \varphi_h$
     then shows that $\varphi$ is a $v$-norm. The theorem is thus proved
     if $\rk(\Gamma_F)<\infty$.

     For the rest of the proof, assume that $\Gamma_F$ has infinite
     rank. Let $(a_{i})_{i=1}^n$ be an $F$-base of $A$. Write
     $a_{i}a_{k}=\sum_{l}c_{ikl}a_{l}$ for some $c_{ikl}\in F$ and
     $\sigma(a_{i})=\sum_{k}d_{ik}a_{k}$ for some $d_{ik}\in F$. Let $F_{0}$
     be the prime subfield of~$F$, and let $F_{1}=F_{0}(\{c_{ikl},d_{ik}\mid
     1\leq i,k,l\leq n\})\subseteq F$. Let $A_{1}$ be the $F_{1}$-span of
     the $a_{i}$, which is an $F_1$-algebra. 
We have $A_{1}\otimes_{F_{1}}F=A$ and $\sigma$ restricts
     to an involution $\sigma_{1}$ on $A_{1}$.
Now, let $(e_i)_{i=1}^n$
be a splitting base of $A_h$ for  the $v$-norm $\varphi_h$.  We need to
enlarge $F_1$ to capture the $e_i$ in the Henselization:  let $L$ be any
field with $F_1\subseteq L \subseteq F$ and $L$ finitely generated over
$F_1$, and let $v_L = v\rest L$.  Since $F_h$~is Henselian, there is a
unique Henselization $(L_h, v_{L,h})$ of $(L, v_L)$ inside $(F_h, v_h)$  by
\cite[Th.~5.2.2(2), p.~121]{EP}.
Because $F$~is the direct limit of such fields~$L$,
the direct limit over such $L$ of the $(L_h,v_{L,h})$ is a
Henselian valued field $(M, v_M)$ with $F \subseteq M\subseteq F_h$ and
$v_h\rest{M} = v_M$.  Therefore, $(M, v_M) = (F_h, v_h)$ by the
uniqueness of  the  Henselization.
Since $A_h = A_1\otimes _{F_1}M$,
there is a field $F_2$ finitely generated over~$F_1$ (hence also over
$F_0$) such that
$e_1$, \ldots, $e_n\in A_1\otimes _{F_1}(F_2)_h$.  Let
\[
A_2 =
A_1\otimes_{F_1}F_2\subseteq A,\qquad
A_{2,h} = A_1\otimes _{F_1}(F_2)_h= A_2\otimes_{F_2} (F_2)_h\subseteq
A_h.
\]
Note that $A_2$ is a simple
$F_2$-algebra since $A = A_{F_2}\otimes_{F_2}F$ and $A$ is simple.
Let $\sigma _2 = \sigma\rest {A_2}$, which is an involution on $A_2$,
and let $\sigma_{2,h} = \sigma_h\rest {A_{2,h}} =
\sigma_2 \otimes \id_{(F_2)_h}$, which is an anisotropic involution on~
$A_{2,h}$.
Let $\varphi_2 = \varphi_h\rest {A_2}$ and $\varphi_{2,h} = \varphi_h\rest
{A_{2,h}}$.  Since $A_h = A_{2,h} \otimes _{(F_2)_h}F_h$ and $e_1$, \ldots,
$e_n\in A_{2,h}$,
Lemma~\ref{reconstruct} says that  $\varphi_{2,h}$ is a $v_{F_2,h}$-norm
on $A_{2,h}$ and $\varphi_h = \varphi_{2,h} \otimes v_h$.  Now, $F_2$
is finitely generated over the prime field $F_0$, so
$\rk(\Gamma_{F_2,v_{F_2}}) \le \trdeg(F_2/F_0) <\infty$ by \cite[\S10.3,
     Cor.~2]{B}.
Since $\varphi_{2,h}$
is a $v_{F_2,h}$-norm, the finite rank case shows that $\varphi_2$ is a
$v_{F_2}$-norm on $A_2$; then,  $\varphi_{2,h} = \varphi_2\otimes v_{F_2,h}$
by Cor.~\ref{immediate.lem}.
Hence,
\[
\varphi_2 \otimes v_h  \ = \  (\varphi_2 \otimes v_{F_2,h}) \otimes v_h \ = \
\varphi_{2,h} \otimes v_h \ = \ \varphi_h .
\]
Therefore, as $\varphi_2$ is a norm, $\varphi_h\rest A =
(\varphi_2 \otimes v_h)\rest{A_2 \otimes_{F_2}F} = \varphi_2 \otimes v$,
which is a norm since it is a scalar extension of the norm
$\varphi_2$.
\end{proof}

\begin{corollary}\label{mainthcor} 
With the hypotheses on $A$, $\sigma$, and $v$ as in Th.~\ref 
{mainnonHensel.thm}, let $\varphi$ be a $v$-gauge on $A$ which
is invariant under $\sigma$.  Then,
\begin{enumerate}
    \item[(a)]
If the residue involution $\sigma_0$ is anisotropic, then
$\varphi$ is the unique $\sigma$-special $v$-gauge on $A$.
   \item[(b)]
If $\sigma_0$ is isotropic, then there is no $\sigma$-special
$v$-gauge  on $A$.
\end{enumerate}
\end{corollary}

\begin{proof}
Let $(F_h,v_h)$ be a Henselization of $(F,v)$, and let
$A_h = A\otimes _F F_h$ and $\sigma_h = \sigma \otimes \id_{F_h}$. 
Let $\varphi_h = \varphi \otimes v_h$,
a surmultiplicative value function  on $A_h$ which is invariant under 
the involution~$\sigma_h$, by Cor.~\ref{cor:compatext}. The graded
isomorphisms
$\gr_{\varphi_h}
(A_h)\cong \gr_\varphi(A)\otimes _{\gr_v(F)} \gr_{v_h}(F_h)
\cong \gr_\varphi(A)$ show that $\varphi_h$~is a gauge on
$A_h$, and $\tilde\sigma_h\cong \tilde\sigma$ and
$(\sigma_h)_0\cong \sigma_0$. (a) If $\sigma_0$ is anisotropic,
then so is $(\sigma_h)_0$, and so also is $\sigma_h$
by Cor.~\ref{anisot.cor}.  Th.~\ref{mainnonHensel.thm} then shows 
that $\varphi$ is a $\sigma$-special $v$-gauge
and is the unique such $v$-gauge on $A$, proving (a).
For (b), we prove the contrapositive:  If there were a
$\sigma$-special $v$-gauge $\psi$ for $A$ then the uniqueness
in Th.~\ref{mainnonHensel.thm} shows that $\varphi = \psi$.
Hence, $\sigma_h$ is anisotropic by Th.~\ref{mainnonHensel.thm},  so 
$(\sigma_h)_0$ is anisotropic
by Cor.~\ref{anisot.cor}, which implies $\sigma_0$~
is anisotropic as well.
\end{proof}

\begin{example} Even when there is no $\sigma$-special
$v$ gauge on $A$, there
may still be tame $v$-gauges on $A$ invariant under~$\sigma$, but they
need not be unique. For example, let $A$ the quaternion
division algebra $(-1,-1)_{\mathbb{Q}}$ over the field of rational
numbers, and let $v$  be the $3$-adic valuation on $\mathbb{Q}$.
Let
$(1,i,j,k)$ be the quaternion base of $A$ with $i^2=j^2=-1$ and
$k=ij=-ji$. As shown in \cite[Ex.~1.16]{TWgr}, a $v$-gauge $\varphi$
can be defined on $A$ by
\[
\varphi(a_0+a_1i+a_2j+a_3k) \ = \ \min\bigl(v(a_0),v(a_1),v(a_2),v(a_3)\bigr).
\]
Clearly, the residue algebra of $A$ for $\varphi$ is
$A_0 = (-1,-1)_{\mathbb{F}_3} \cong M_2(\mathbb{F}_3)$.
The $v$-gauge $\varphi$ is obviously invariant under the conjugation 
involution $\sigma$ on $A$. (This is the involution
with $\sigma(i) = -i$ and ${\sigma(j) = -j}$, which is the unique
symplectic involution on $A$.)  Since $A$ is a division algebra, 
$\sigma$ must be anisotropic.  The residue
involution $\sigma_0$ is the conjugation involution on
$A_0$, which is  isotropic, since $\sigma_0(t)t
=\Nrd_{A_0}(t)$ for any $t$ in the split quaternion algebra $A_0$. 
So, by Cor.~\ref{mainthcor}(b) there is no
$\sigma$-special $v$-gauge on $A$.
For
any unit $u\in A^\times$ the map $\varphi_u$ defined by
\[
\varphi_u(x) \,= \,\varphi(uxu^{-1})\qquad\text{for $x\in A$}
\]
is a $v$-gauge on $A$, and
Prop.~1.17 of \cite{TWgr} shows that $\varphi_u=\varphi$ if and only
if $\widetilde u$ is invertible in $\gr_\varphi(A)$, which is not a 
graded division ring.  But, for {\it every}
$u\in A^\times$, since $\varphi$ is invariant under $\sigma$ and 
$\sigma(u)u$ is central,
$$
\varphi_u(\sigma(x)) \ = \ \varphi(\sigma[(\sigma(u^{-1})x\sigma(u)]) 
\ = \ \varphi(\sigma(u^{-1})x\sigma(u)) \ = \
\varphi_u([\sigma(u)u]^{-1}x [\sigma(u)u]) \ = \
\varphi_u(x),
$$
showing that $\varphi_u$ is invariant under $\sigma$.
\end{example}

\end{document}